\documentclass[reqno, 11pt]{amsart}

\usepackage{color}
\usepackage{mathrsfs}
\usepackage{amscd}
\usepackage{amsmath}
\usepackage{latexsym}
\usepackage{amsfonts}

\usepackage{amssymb}
\usepackage{amsthm}
\usepackage{graphicx}
\usepackage{hyperref}
\usepackage{makecell}
\usepackage{array,color}
\usepackage{booktabs}
\usepackage{multirow}

\newtheorem{theorem}{Theorem}[section]
\newtheorem{proposition}[theorem]{Proposition}
\newtheorem{corollary}[theorem]{Corollary}
\newtheorem{lemma}[theorem]{Lemma}
\newtheorem{definition}[theorem]{Definition}

\newtheorem{remark}[theorem]{Remark}

\newcommand{\norm}[1]{{\|#1\|}}

\newcommand\R{\mathbb{R}}

\def\sgn{\mathop{\mathrm{sgn}}\nolimits}
\def\Id{\mathop{\mathrm{Id}}\nolimits}
\def\<{{\langle}}
\def\>{{\rangle}}

\numberwithin{equation}{section}

\title[Decay estimates for beam equations with potentials on the line]
{Decay estimates for beam equations with potentials on the line}

\author{Shuangshuang Chen, Zijun Wan and Xiaohua Yao\textsuperscript{\dag}}

\address{Shuangshuang Chen, Department of Mathematics, Central China Normal University, Wuhan, 430079, P.R. China}
\email{chenss@mails.ccnu.edu.cn}

\address{Zijun Wan, Department of Mathematics, Central China Normal University, Wuhan, 430079, P.R. China}
\email{zijunwan@mails.ccnu.edu.cn}

\address{Xiaohua Yao, Department of Mathematics and Key Laboratory of Nonlinear Analysis and Applications(Ministry of Education), Central China Normal University, Wuhan, 430079, P.R. China}
\email{yaoxiaohua@ccnu.edu.cn}

\thanks{\textsuperscript{\dag} Corresponding author}

\date{\today}

\begin{document}
	
	\begin{abstract}\baselineskip=13pt
		This paper is devoted to the time decay estimates for  the following beam equation with a potential:
		\begin{equation*}
			\begin{cases}
				\partial_t^2 u + \left( \Delta^2 + m^2 + V(x) \right) u = 0, & (t, x) \in \mathbb{R} \times \mathbb{R}, \\
				u(0, x) = f(x), \quad \partial_t u(0, x) = g(x), &
			\end{cases}
		\end{equation*}
		where $V$ is a real-valued decaying potential on $\mathbb{R}$, and $m \in \mathbb{R}$.
		
		Let $H = \Delta^2 + V$ and  $P_{ac}(H)$ denote the projection onto the absolutely continuous spectrum of $H$.
		Then  for $m = 0$, we establish the following decay estimates of the solution operators:
		$$
		\left\|\cos (t \sqrt{H}) P_{ac}(H)\right\|_{L^1 \rightarrow L^{\infty}} +
		\left\|\frac{\sin (t \sqrt{H})}{t \sqrt{H}} P_{ac}(H)\right\|_{L^1 \rightarrow L^{\infty}} \lesssim |t|^{-\frac{1}{2}}.
		$$
		But for $m \neq 0$, the solutions have different time decay estimates from the case  where  $m=0$. Specifically, the $L^1$-$L^\infty$ estimates of $\cos (t \sqrt{H + m^2})$ and $\frac{\sin (t \sqrt{H + m^2})}{\sqrt{H + m^2}}$ are bounded by $O(|t|^{-\frac{1}{4}})$ in the low-energy part  and $O(|t|^{-\frac{1}{2}})$ in the high-energy  part.
		
		It is noteworthy that all  these results remain consistent with the  free cases (i.e., $V = 0$)  whatever zero is a regular point or a resonance of $H$.
		As consequences, we establish the corresponding Strichartz estimates, which are fundamental to study nonlinear problems of beam equations.
	\end{abstract}
	\maketitle
	
	\baselineskip=15pt
	\section{Introduction and main results}
	\subsection{Introduction}
	In this paper, we investigate the following beam equation with a real-valued decaying potential in one dimension:
	\begin{equation}\label{beam_equation}
		\begin{cases}
			\partial_t^2u+ \left( \Delta^2 + m^2 + V(x) \right) u = 0, & (t, x) \in \mathbb{R} \times \mathbb{R}, \\
			u(0, x) = f(x), \quad \partial_tu(0, x) = g(x), &
		\end{cases}
	\end{equation}
	where the potential $ V(x)$ satisfies  $ |V(x)| \lesssim \langle x \rangle^{-\mu}$ for some $\mu > 0 $ and  $m \in \mathbb{R}$.
	As we know, the equation was employed to model the bending of beams under various loading conditions  in dynamic way.
	In  the equation \eqref {beam_equation},  $u(t, x)$ describes the displacement of the beam at time $ t $ and position $ x $,  and the operator $ \Delta^2 + m^2 + V$ captures both the spatial bending behavior and external forces acting on the system.
	

	Let $ H= \Delta^2 + V $.
By  Kato-Rellich theorem,  as the operator sum of $\Delta^2$ and  $V$,   $H$ is   a self-adjoint operator on $L^2(\mathbb R)$ with domain $D(H)=D(\Delta^2)=H^4(\R)$.
	The solution to the beam equation \eqref{beam_equation} can be explicitly expressed as
	\begin{equation}\label{solution_m}
		u_m(t, x) = \cos (t \sqrt{H + m^2}) f(x) + \frac{\sin (t \sqrt{H + m^2})}{\sqrt{H + m^2}} g(x).
	\end{equation}
	The expression \eqref{solution_m} above depends on the branch choice of $ \sqrt{z} $ with $\Im\sqrt{z} \geq 0 $, ensuring that the solution $u(t,x) $ is well-defined even if $H $ is not positive. This work focuses on the decay estimates of the propagator operators $ \cos(t\sqrt{H+m^2})$ and $ \frac{\sin(t\sqrt{H+m^2})}{\sqrt{H+m^2}}$.
	
	It is well known that the spectrum of
	$H$ consists of  a discrete set of negative eigenvalues $\left\{\lambda_1 \leq \lambda_2 \leq \cdots<0\right\}$ and absolutely continuous spectrum $[0, \infty)$ with possible embedded eigenvalues, provided that the potential
	$V$ satisfies certain decay conditions, see, e.g.,  \cite{Agmon, RS}.
	
	Throughout the paper, we always assume that $H$ has no embedded positive eigenvalues (see Subsection \ref{eigenvalue} below for some sufficient condition).
	Under this assumption, we denote the negative eigenvalues of
	$H$ by $\lambda_j(j \geq 1)$,
	counting multiplicities of every $\lambda_j$, where
	$H \phi_j=\lambda_j \phi_j ( j \geq 1 )$ for $\phi_j \in L^2\left(\mathbb{R}\right)$, and let $P_{a c}(H)$ denote the projection onto the absolutely continuous spectrum space of $H$.
	Given these notations, the solution \eqref{solution_m} of the Beam equation
	\eqref{beam_equation}  can then be expressed as
	$$
	u_m(t,x):=u_{m,d}(t, x)+u_{m,c}(t, x),
	$$
	where
	$$
	\begin{aligned}
		& u_{m,d}(t,x)=\sum_{j} \cosh( t\sqrt{-\lambda_j-m^2})(f,\phi_j)\phi_j(x) +\frac{\sinh (t\sqrt{-\lambda_j-m^2})}{\sqrt{-\lambda_j-m^2}}(g,\phi_j)\phi_j(x), \\
		& u_{m,c}(t,x)=\cos (t \sqrt{H+m^2})P_{ac}(H)f(x) + \frac{\sin (t \sqrt{H+m^2})}{\sqrt{H+m^2}}P_{ac}(H)g(x).
	\end{aligned}
	$$
	The presence of negative eigenvalues in the spectrum of
	$H$ leads to exponential growth in
	$u_{m,d}(t,x)$, when $m^2<-\lambda_j$ for some $j$ and   $t$ becomes large. As a result, it is crucial to remove the discrete component associated with these negative eigenvalues and focus on the decay estimates for the continuous part $u_{m,c}(t, x)$.
	
	For the free case (i.e., $V=0$) in one dimension, the solution operators satisfy the following optimal time decay estimates:
	\begin{equation}\label{free estimate}
		\|\cos (t \Delta)\|_{L^{1}\rightarrow L^{\infty}}+\left \|\frac{\sin (t \Delta)}{t\Delta}\right \|_{L^1 \rightarrow L^{\infty}}  \lesssim|t|^{-\frac{1}{2}},
	\end{equation}
	and  e.g.  for $m=1,$
	\begin{align}\label{free nonhom-estimate}
		&\left\|\cos (t \sqrt{\Delta^2+1})\chi_j(-\Delta)\right\|_{L^1 \rightarrow L^{\infty}} + \left \|\frac{\sin (t \sqrt{\Delta^2+1})}{\sqrt{\Delta^2+1}} \chi_j(-\Delta)\right \|_{L^1 \rightarrow L^{\infty}}\lesssim
		\begin{cases}
			|t|^{-\frac{1}{4}}, & \text{if } j=1, \\
			|t|^{-\frac{1}{2}}, & \text{if } j=2,
		\end{cases}
	\end{align}
	where $\chi_1(\lambda)$ is a smooth cut-off function supported near zero and $\chi_2(\lambda) := 1-\chi_1(\lambda)$, see Theorem \ref{main_theorem} below, also see \cite{CMY}. The essential reason leading to the difference between the estimates \eqref{free estimate} and \eqref{free nonhom-estimate} is that the nonlocal operator $\sqrt{\Delta^2+1}$ behaves as $1+\frac{1}{2}\Delta^2$ in low frequent part, and like $-\Delta$ in high frequent part.
	
	The primary objective of this paper is to establish time decay estimates for the solution of the  beam equation
	\eqref{beam_equation} with a potential $V$ and parameter  $m\in\R$, and derive  specific Strichartz estimates.  Remarkably, all our results are consistent with the corresponding free cases (i.e., $V=0$), whatever zero is a regular point or a resonance of $H$.
	To achieve this, we begin with  Stone's formula, Littlewood-Paley decomposition and oscillatory integral theory to obtain the results of the free case. Subsequently, we
	analyze the asymptotic expansions of the resolvent $R_V(\lambda^4)$ for $\lambda$ near zero, taking into account the presence of resonances. Then, we establish the desired time decay bounds for all types of resonances.

	\subsection{Notations}\label{Notations}
	In this subsection, we collect some notations used throughout the paper.
	\begin{itemize}
		\item
		For $a\in \mathbb{R}$, $a\pm$ denote $a\pm \epsilon$ for any small $\epsilon >0$. $\left \lfloor  a\right \rfloor$ denotes the largest integer less than or equal to $a$.
		\vskip0.2cm
		\item For $a, b\in \mathbb{R}^{+}$, $a\lesssim b$ (resp. $a\gtrsim b$) means $a\le cb$ (resp. $a\ge cb$) with some constant $c>0$. Moreover, we denote $a \lesssim_m b$ if the constant depends on the variable $m$.
		\vskip0.2cm
		\item Let $\varphi $ be a smooth function such that $\varphi(s)=1$ for
		$|s| \leq \frac{1}{2}$ and $\varphi(s)=0$ for $|s| \geq 1$. $\varphi_N(s):=\varphi\left(2^{-N} s\right)-\varphi\left(2^{-N+1} s\right), N \in \mathbb{Z}$. Then $\varphi_N(s)=\varphi_0\left(2^{-N} s\right), \operatorname{supp} \varphi_0 \subset$ $\left[\frac{1}{4}, 1\right]$ and
		\begin{equation}\label{varphi_0}
			\sum_{N=-\infty}^{\infty} \varphi_0\left(2^{-N} s\right)=1,\quad  s \in \mathbb{R} \backslash\{0\}.
		\end{equation}
		\vskip0.2cm
		\item Let $\chi_1 \in C_c^{\infty}(\mathbb{R})$ such that $\chi_1(\lambda)=1$ for $|\lambda|<\lambda_0\ll1$ and $\chi_1(\lambda)=0$ if $|\lambda|>2\lambda_0$, where $\lambda_0$ is some sufficiently small positive constant depending on low energy expansion of $\left( M^{\pm}(\lambda)\right)^{-1}$ in Theorem \ref{lem_M}. And denote $\chi_2(\lambda):=1-\chi_1(\lambda)$.
		\vskip0.2cm
		\item A bounded operator $K \in \mathbb B(L^2)$ (with kernel $K(x,y)$) is said to be absolutely bounded ($K \in \mathbb{AB}(L^2)$ for short) if the integral operator $|K|$ with kernel $|K(x,y)|$ is also bounded on $L^2$. Note that $K \in \mathbb{AB}(L^2)$ if  $K$ is a Hilbert--Schmidt operator on $L^2$.
		\vskip0.2cm
		\item 	Let $\langle \cdot \rangle = 1 + | \cdot |$.
		For $s \in \mathbb{R}$, we  define the weighted $L^2 $ spaces:
		$$L^2_s(\mathbb{R}) := \{ f \in L_{\text {loc }}^2(\mathbb{R}) \mid  \langle \cdot \rangle^s f \in L^2(\mathbb{R}) \} .$$
	\end{itemize}
	
	\subsection{Main results}\label{main result}
	Before giving our results, we need to recall the notions of the zero resonances for the operator $H=\Delta^2+V(x)$ on $\mathbb{R}$ due to Soffer-Wu-Yao \cite{SWY22}.  We define
	$$
	W_\sigma(\mathbb{R})=\bigcap_{s>\sigma} L_{-s}^2(\mathbb{R})
	$$
	which is increasing in $\sigma\in\R$ and satisfies $L_{-\sigma}^2(\mathbb{R}) \subset W_\sigma(\mathbb{R})$. Note that $\langle x \rangle^\alpha  \in W_\sigma(\mathbb{R})$ if $\sigma \geq \alpha+\frac{1}{2}$. In particular, $f \in W_{\frac{1}{2}}(\mathbb{R})$ and $\langle x\rangle f \in W_{\frac{3}{2}}(\mathbb{R})$ for any $f \in L^{\infty}(\mathbb{R})$.
	\begin{definition}
		\label{definition_resonance} Let $H=\Delta^2+V(x)$ and $|V(x)|\lesssim \left \langle  x\right \rangle^{-\mu}$ for some $\mu>0$.
		We say that
		\begin{itemize}
			\item[(i)] zero is the first kind resonance  of $H$ if there exists some nonzero $\phi\in W_{\frac{3}{2}}(\mathbb{R})$ but no any nonzero $\phi\in W_{\frac{1}{2}}(\mathbb{R})$ such that $H\phi=0$ in the distributional sense;
			\item[(ii)] zero is the second kind resonance of $H$ if there exists some nonzero $\phi\in W_{\frac{1}{2}}(\mathbb{R})$ but no any nonzero $\phi\in L^2(\mathbb{R})$ such that $H\phi=0$ in the distributional sense;
			\item[(iii)]  zero is an eigenvalue of $H$ if there exists  some nonzero $\phi\in L^2(\mathbb{R})$ such that $H\phi=0$ in the distributional sense;
			\item[(iv)]  zero is a regular point of $H$ if $H$ has neither zero eigenvalue nor zero resonances.
		\end{itemize}
	\end{definition}
	
		It was observed by Goldberg \cite{Goldberg} (see also Remark 1.2 and Remark 3.6  in \cite{SWY22}) that when a potential $V(x)$  decays polynomially,  zero eigenvalue in the one-dimensional fourth order Schr\"odinger operator may not occur, see, e.g., \cite{DeTr} for Schr\"odinger operator case. Specifically, if the potential $V$ satisfies the decay condition of Theorem \ref{main_theorem}, then $H$ has no zero eigenvalue (see Lemma 3.5 in \cite{SWY22}). Therefore, the zero eigenvalue case will be excluded in the subsequent theorems.
	Now we state our main results.
	
	\begin{theorem}\label{main_theorem}
		Let $|V(x)| \lesssim \left \langle  x\right \rangle^{-\mu}$ with some $\mu>0$ depending on the following zero energy types:
		\begin{equation}\label{condition}
			\mu> \begin{cases}13, & \text { zero is a regular point, } \\ 17, & \text { zero is the first kind resonance, } \\ 25, & \text { zero is the second kind resonance. }\end{cases}
		\end{equation}
		Assume that $H=\Delta^2+V$ has no positive embedded eigenvalue and let  $P_{a c}(H)$ denote the projection onto the absolutely continuous spectrum space of $H$. Then
		\begin{equation}\label{homoge-case}
			\left\|\cos (t \sqrt{H}) P_{ac}(H)\right\|_{L^1 \rightarrow L^{\infty}}+
			\left \|\frac{\sin (t \sqrt{H})}{t\sqrt{H}} P_{a c}(H)\right \|_{L^1 \rightarrow L^{\infty}} \lesssim|t|^{-\frac{1}{2}}.
		\end{equation}
		Moreover, let $ \Psi_j(H)=P_{ac}(H) \chi_j(H), j=1,2$, then for $m \neq 0,$
		$$
		\begin{aligned}
			&\left\|\cos (t \sqrt{H+m^2}) \Psi_j(H)\right\|_{L^1 \rightarrow L^{\infty}} \lesssim
			\begin{cases}
				\left\langle  m \right\rangle^{\frac{1}{2}}|t|^{-\frac{1}{4}}, & \text{if}\ j=1, \\
				\left\langle  m \right\rangle |t|^{-\frac{1}{2}}, & \text{if}\  j=2;
			\end{cases} \\
			& \left \|\frac{\sin (t \sqrt{H+m^2})}{\sqrt{H+m^2}} \Psi_j(H)\right \|_{L^1 \rightarrow L^{\infty}}\lesssim
			\begin{cases}
				m^{-1}\left\langle  m \right\rangle^{\frac{1}{2}}|t|^{-\frac{1}{4}}, & \text{if}\  j=1, \\
				(1+\frac{1}{m}) |t|^{-\frac{1}{2}}, & \text{if}\ j=2.
			\end{cases}
		\end{aligned}
		$$
		As a result, for $m\neq 0$ we have
		\begin{equation}\label{nonhomogeneous}
			\left\|\cos (t \sqrt{H+m^2})P_{ac}(H) \right\|_{L^1 \rightarrow L^{\infty}}+\left \|\frac{\sin (t \sqrt{H+m^2})}{\sqrt{H+m^2}} P_{ac}(H)\right \|_{L^1 \rightarrow L^{\infty}}\lesssim_m
			\begin{cases}
				|t|^{-\frac{1}{2}}, & \text{for}\  |t|\le 1;\\
				|t|^{-\frac{1}{4}}, & \text{for}\  |t|>1.
			\end{cases}
		\end{equation}
	\end{theorem}
	\begin{remark}\label{remark_1}A few of  remarks on Theorem \ref{main_theorem} are given as follows:
		
		\begin{itemize}
			\item
				We observe that the decay assumption \eqref{condition} on the potential $V$ arises from the low-energy expansions of the resolvent $R^\pm_V(\lambda^4)=(H - \lambda^4 \mp i0)^{-1}$  by Soffer-Wu-Yao \cite{SWY22} (see Theorem \ref{lem_M}),   while the high energy part  only requires the lower decay rate
			$\mu>2$ (see Theorem \ref{main_theorem_high}).  Hence the potential assumption \eqref{condition} may not be optimal, even in the regular case. However,
			the decay estimates \eqref{homoge-case} and \eqref{nonhomogeneous} for the case $ V \neq 0 $ are actually optimal, as they remain consistent with those of the free case.
			\vskip0.2cm
			\item   The solution operators $ \frac{\sin(t\sqrt{H})}{\sqrt{H}}$ and $ \frac{\sin(t\sqrt{H+m^2})}{\sqrt{H+m^2}}$ have different $L^1$-$L^\infty$ estimates in the low energy part. The former behaves as the time growth $O(|t|^{1/2})$ as $|t| \to \infty $ due to the additional singular factor $\frac{1}{\sqrt{H}}$, while the latter has the time decay $O(|t|^{-1/4})$ for any $m\neq 0$.
\vskip0.2cm
			\item
			We may point out that the time decay rate may change for various types of dispersive estimates when a zero resonance or eigenvalue occurs. For example, see \cite{JK, ES1} for three-dimensional Schr\"odinger operators $-\Delta+V$ and \cite{Erdogan-Green} for two-dimensional Schr\"odinger operators. In the case of the classical wave equation with potential (see \eqref{waveequation-twoorder} below), it is essential to note that a first-kind zero resonance does not alter the decay rate, while a second-kind resonance can diminish it, as shown in \cite{Green14} for two dimensions and \cite{EGG14} for four dimensions.
			\vskip0.2cm
			\item In the context of fourth-order Schr\"odinger operators $\Delta^2+V$, the impact of zero resonances has been considered by \cite{LSY21, EGT19, GT19}. In particular,  Chen et al.\cite{CLSY_ArXiv} was the first to  prove that the $L^1$-$L^\infty$ estimates for the solution operators  $\cos(t\sqrt{H})$ and $
			\frac{\sin(t\sqrt{H})}{t\sqrt{H}}
			$ are $O(|t|^{-3/2})$ in dimension three when zero is a regular point or a first-kind resonance. Furthermore, they showed that the decay rate changes when other types of zero-energy resonances occur.
			In comparison to the work of \cite{CLSY_ArXiv}, our results in  dimension one demonstrate that the estimates remain independent of whether zero is a regular point or a resonance of $H$.
		\end{itemize}
	\end{remark}
	Note that the operators $\cos (t \sqrt{H}) P_{ac}(H)$ and $\frac{\sin(t\sqrt{H})}{t \sqrt{H}} P_{ac}(H) $ are uniformly bounded on $L^2$. Hence by  Riesz-Thorin interpolation and the decay estimate \eqref{homoge-case} of Theorem \ref{main_theorem}, we immediately have the following corollary.
	\begin{corollary}
		Let $H = \Delta^2 + V$ satisfy the conditions specified in Theorem \ref{main_theorem}. Then, for all $1 \leq p \leq 2$ and $\frac{1}{p} + \frac{1}{p'} = 1$, the following  $L^p\to L^{p^{\prime}}$ estimates hold:
		\begin{equation}\label{Lp-estimate}
			\left\|\cos(t \sqrt{H}) P_{ac}(H) f\right\|_{L^{p'}} + \left\|\frac{\sin(t\sqrt{H})}{t \sqrt{H}} P_{ac}(H) f\right\|_{L^{p'}} \lesssim |t|^{-\left(\frac{1}{p} - \frac{1}{2}\right)} \|f\|_{L^p}.
		\end{equation}
	\end{corollary}

	Recently, Mizutani-Wan-Yao \cite{MWY22} showed that the following wave operators
	$$
	W_\pm=W_\pm(H,\Delta^2) :=s-\lim_{t\to\pm\infty}e^{itH}e^{-it\Delta^2}
	$$
	are bounded on $L^p(\mathbb{R})$ for $1<p<\infty$ if $H=\Delta^2+V$ satisfying $|V(x)|\lesssim \langle x \rangle ^{-\mu}$ for some $\mu>0$.
	Note that $W_\pm$ satisfy the intertwining identity
	$f(H)P_{ac}(H)=W_\pm f(\Delta^2)W_\pm^*$,
	where $f$ is any Borel measurable function on $\mathbb{R}$. By the $L^p$-boundedness of $W_\pm$ and $W_\pm^*$, one can tranfer the $L^p$-$L^{p'}$ estimates of $f(H)P_{ac}(H)$ to the same estimates of $f(\Delta^2)$ by
	\begin{equation}\label{wave operator methods}
		\|f(H)P_{ac}(H)\|_{L^p\to L^{p'}}\le \|W_\pm\|_{L^{p'}\to L^{p'}}\ \|f(\Delta^2)\|_{L^p\to L^{p'}}\ \|W_\pm^*\|_{L^{p}\to L^{p}}.
	\end{equation}
	Let $f(\lambda)=\cos(t\sqrt{\lambda})$ and $\frac{\sin(t\sqrt{\lambda})}{t\sqrt{\lambda}}$. Then for any $1<p\le 2$, the $L^p$-$L^{p'}$ estimates \eqref{Lp-estimate} can be obtained from  the \eqref{wave operator methods} and the free estimate \eqref{free estimate} of  $\cos(t\Delta)$ and $\frac{\sin(t\Delta)}{t\Delta}$.
	However, note that wave operators $W_\pm$ are unbounded on the endpoint spaces  $L^1$ and $L^\infty$ even for compactly supported potentials $V$ (see \cite{MWY22}), therefore  the time decay estimates in Theorem \ref{main_theorem} can not be obtained by wave operator methods.
	\vskip0.3cm
	Finally, by using the $L^1$-$L^\infty$ estimates in Theorem \ref{main_theorem} and applying abstract arguments in Keel and Tao \cite{Kee98},  we will
	establish Strichartz estimates for the solution of Beam equations \eqref{beam_equation}. Recall first that the exponent pair $(q, r)$ is said to be $\sigma$-\emph{admissible} if $q, r \geq 2$, $(q, r, \sigma) \neq (2, \infty, 1)$, and
	\begin{equation}\label{admissible}
		\frac{1}{q} + \frac{\sigma}{r} \leq \frac{\sigma}{2}.
	\end{equation}
	If the equality holds in \eqref{admissible}, we refer to $(q, r)$ as \emph{sharp} $\sigma$-\emph{admissible}; Otherwise, it is termed \emph{nonsharp} $\sigma$-\emph{admissible}. Moreover, the solution to \eqref{beam_equation} can be represented as
	$$
		u_m(t, x) = \cos(t \sqrt{H+m^2}) f(x) + \frac{\sin(t \sqrt{H+m^2})}{\sqrt{H+m^2}} g(x),\ \ m\in \R.
	$$
	
	\begin{theorem}
		Let $H = \Delta^2 + V$ satisfy the conditions specified in Theorem \ref{main_theorem} and  let $(q,r), (\widetilde{q},\widetilde{r})$ be sharp $\frac{1}{2}$-admissible, then the following estimates hold:
	\begin{itemize}
		\item[(i)] The homogeneous Strichartz estimate
			$$
			\left\| e^{\pm it\sqrt{H}}P_{ac}(H)f\right\|_{L_t^q L_x^r} \lesssim \left\| f\right\|_{L^2},
			$$
			which gives that
			$$
			\left\|P_{a c}(H)u_0\right \|_{L_t^q L_x^r} \lesssim \left\| f\right\|_{L^2}+\left\| \frac{P_{ac}(H)}{\sqrt{H}}g\right\|_{L^2};
			$$
			\item[(ii)] The dual homogeneous Strichartz estimate
			$$
			\begin{aligned}
				\left\|\int_{\mathbb{R}} e^{\pm it\sqrt{H}}P_{ac}(H)F(s,x) ds\right\|_{L_x^2} \lesssim \left\| F\right\|_{L_t^{q^\prime} L_x^{r^\prime}};
			\end{aligned}
			$$
			\item[(iii)] The retarded Strichartz estimate
			$$
			\left\| \int_{s<t}e^{\pm i(t-s)\sqrt{H}}P_{ac}(H)F(s,x) ds\right\|_{L_t^{q} L_x^{r}} \lesssim \left\| F\right\|_{L_t^{\widetilde{q}^\prime} L_x^{\widetilde{r}^\prime}}.
			$$
	\end{itemize}
\end{theorem}
\begin{theorem}
		Let $H = \Delta^2 + V$ satisfy the conditions specified in Theorem \ref{main_theorem} and  let $(q_1,r_1), (\widetilde{q}_1,\widetilde{r}_1)$ be sharp $\frac{1}{4}$-admissible and $(q_2,r_2), (\widetilde{q}_2,\widetilde{r}_2)$ be sharp $\frac{1}{2}$-admissible. Denote $\Psi_j(H):=P_{ac}(H)\chi_j(H), j=1,2$, then for $m \neq 0$, we have
		\begin{itemize}
			\item[(i)] The homogeneous Strichartz estimate
			$$
			\left\| e^{\pm it\sqrt{H+m^2}}\Psi_j(H) f\right\|_{L_t^{q_j} L_x^{r_j}} \lesssim_m \left\| f\right\|_{L^2},
			$$
			which gives that
			$$
			\begin{aligned}
				\left\|\Psi_j(H)u_m\right \|_{L_t^{q_j} L_x^{r_j}}
				\lesssim_m \left\| f\right\|_{L^2}+\left\| g\right\|_{L^2};
			\end{aligned}
			$$
			\item[(ii)] The dual homogeneous Strichartz estimate
			$$
			\begin{aligned}
				& \left\|\int_{\mathbb{R}} e^{\pm it\sqrt{H+m^2}}\Psi_j(H)F(s,x) ds\right\|_{L_x^2} \lesssim_m \left\| F\right\|_{L_t^{{q}_j^\prime} L_x^{{r}_j^\prime}};
			\end{aligned}
			$$
			\item[(iii)] The retarded Strichartz estimate
			$$
			\left\| \int_{s<t}e^{\pm i(t-s)\sqrt{H+m^2}}\Psi_j(H)F(s,x) ds\right\|_{L_t^{q_j} L_x^{r_j}} \lesssim_m \left\| F\right\|_{L_t^{\widetilde{q}_j^\prime} L_x^{\widetilde{r}_j^\prime}}.
			$$
		\end{itemize}
	\end{theorem}

	\subsection{Further remarks and backgrounds}
	Here, we make further comments, particularly regarding the spectral assumptions, and highlight some known results on higher-order Schr\" odinger operators and second-order wave equations with real-valued decaying potentials.
	\subsubsection{Zero resonance and zero eigenvalue}
	We first present two simple examples of potentials $V $ such that $ H $ has a zero resonance. Let $ \phi_1 \in C^\infty(\mathbb{R}) $ be a positive function defined by $ \phi_1(x) = c|x| + d $ for $ |x| > 1 $, where $ c, d \geq 0 $ and $(c, d) \neq (0, 0) $. Then $ H \phi_1 = 0 $ if we take
	$V(x) = -\frac{\Delta^2 \phi_1}{\phi_1}$ for $ x \in \mathbb{R}$.
	In this case, $ V \in C_0^\infty(\mathbb{R}) $, and $ \phi_1 \in W_{3/2}(\mathbb{R}) \setminus W_{1/2}(\mathbb{R}) $ if $ c > 0 $, while $ \phi_1 \in W_{1/2}(\mathbb{R}) $ if $c = 0 $. These examples demonstrate that zero resonances may occur even for compactly supported potentials.
	
	We now turn to the case of zero eigenvalues for $H $. It is similarly straightforward to construct an example of \( H \) having a zero eigenvalue if $ V $ decays sufficiently slowly. For instance, let $\phi = (1 + |x|^2)^{-s/2} $ and define $ V(x) = -\frac{\Delta^2 \phi}{\phi} $. Then $\phi \in H^4(\mathbb{R}) $ for any $ s > 1 $, and $(\Delta^2 + V) \phi = 0 $, implying that $|V(x)| \lesssim \langle x \rangle^{-4} $ and zero is an eigenvalue of $ H $. However, if $ |V(x)| \lesssim \langle x \rangle^{-\mu} $ for some $ \mu $ satisfying \eqref{condition}, then zero cannot be an eigenvalue of $ H $ in one dimension. We conjecture that this decay condition on $V $ may not be sharp, and we expect that the optimal decay rate to ensure the absence of zero eigenvalues for $ \Delta^2 + V $ on $ \mathbb{R} $ is $ \mu > 4 $.
	Based on these observations, and considering the polynomial decay conditions imposed on the   potential $V$ in our theorems, we remark that zero eigenvalue can be actually excluded, while zero resonances must be taken into account. However, we again emphasize that the presence of zero resonances does not affect the validity of the time decay estimates derived for the solution of the beam equation  \eqref{beam_equation}.
	\subsubsection{Embedded positive eigenvalue}\label{eigenvalue}
	According to Kato's theorem \cite{Kat59}, if
	$V$ is bounded and satisfies $V=o(|x|^{-1})$ as $|x|\rightarrow\infty$,  then the Schr\"odinger operator $-\Delta+V$ has no positive eigenvalues  (also see \cite{FHHH82, IJ03, KT06} for more related results and references). However, such a criterion does not extend to the fourth-order Schr\"odinger  operator $H=\Delta^2+V$. Indeed, it is easy to construct a Schwartz function $V(x)$ so that $H$ on $\R$ has an eigenvalue $1$.\footnote{In fact, $V(x)=20/\cosh^2(x)-24/\cosh^4(x)\in \mathcal{S}(\R)$ satisfies $\frac{d^4\psi_0}{dx^4}+ V(x)\psi_0=\psi_0$ where $\psi_0=1/\cosh(x)=2/(e^x+e^{-x})\in L^2(\mathbb{R})$.}
	Furthermore, in any dimensions $n\ge1$, one can construct
	$V\in C_0^\infty(\R^n)$ such that $H$ has positive eigenvalues (see Feng et al.\cite[Section 7.1]{FSWY20}).
	These results demonstrate that the absence of positive eigenvalues for the fourth-order Schr\"odinger operator is more subtle and less stable than for the second-order case under perturbations of the potential  $V$.
	
	Nevertheless, it is worth noting that if $V\in C^1(\R)\cap L^\infty(\R)$ is repulsive, {\it i.e.}, $xV'(x)\le0$, then $H$ has no eigenvalues (see \cite[Theorem 1.11]{FSWY20}). This criterion is also applicable to general higher-order elliptic operators $P(D)+V$ in any dimensions $n\ge1$.
	
	\subsubsection{Some results of higher-order Schr\"odinger operators}
	
	Recently, substantial progress has been made in the study of time decay estimates for the operator $ e^{-itH} $ generated by the fourth-order Schr\"odinger operator $ H = \Delta^2 + V $, where $ V $ is a real-valued decaying potential. Feng-Soffer-Yao \cite{FSY} first established that  Kato-Jensen decay estimate for $e^{-itH} $ is bounded by $ (1+|t|)^{-n/4} $ for dimensions $ n \geq 5 $ in the regular case. Subsequently, Erdo\u{g}an-Green-Toprak \cite{EGT19} for $ n=3$, and Green and Toprak \cite{GT19} for $ n=4 $, demonstrated that the $ L^1$-$L^\infty$ estimates for $e^{-itH}$ are $ O(|t|^{-n/4}) $ for $n = 3, 4 $ when zero is a regular point or a first-kind resonance. They also showed that the decay rate changes when other types of zero-energy resonances occur.
	
	More recently, Soffer-Wu-Yao \cite{SWY22} proved that the $ L^1$-$L^\infty$ estimates for $ e^{-itH}$ are $ O(|t|^{-1/4})$ for $ n = 1 $, irrespective of whether zero is a regular point or a resonance. It is worth emphasizing that, in dimension $n=1$, the presence of zero resonances does not affect the optimal decay rate of $e^{-itH}$, though it does require a faster decay rate of the potential.  Furthermore, Li-Soffer-Yao \cite{LSY21} investigated the $ L^1$-$L^\infty $ estimates for $e^{-itH}$ in dimension $n = 2 $ for all types of zero resonances and established an optimal decay rate of $ O(|t|^{-1/2})$ in the regular and first-kind resonance cases.
	Additionally, for the higher-order Schr\"odinger operator $ H = (-\Delta)^m + V $, where with $m > 1 $,  $2m < n < 4m $, and $ V $ is a real-valued potential belonging to the closure of $C_0 $ functions with respect to the generalized Kato norm (which has critical scaling), Erdo\u{g}an-Goldberg-Green \cite{EGG_ArXiv} established that the $ L^1$-$L^\infty $ estimates for $e^{-itH} $ are $ O(|t|^{-\frac{n}{2m}}) $ in the regular case.
	
	Besides, some significant advances have been made in the study of $L^p $ bounds for  wave operators associated with the higher-order Schr\"odinger operator $ H = (-\Delta)^m + V(x)$ on $\mathbb{R}^n $ with $ m > 1 $. Notably, results can be found in \cite{GG21b, MWY22, MWY23, MWY_ArXiv23_2} for $ n < 2m $ with $m = 2 $, in \cite{Yajima_2024JST} for $ n = 2m $ with $ m = 2 $, and in \cite{EG21, Erdogan-Green23, EGG23, EGL} for $n > 2m$ with $m \in \mathbb{N}^+ $.

	\subsubsection{Classical wave equations with potentials}
	
	For the following classical wave equation with a real-valued decaying potential, the decay rate of the solution has been extensively studied:
	\begin{equation}\label{waveequation-twoorder}
		\left\{\begin{array}{l}
			\partial_t^2	u + (-\Delta + V) u = 0, \quad (t,x) \in \mathbb{R} \times \mathbb{R}^n, \\
			u(0,x) = f(x), \quad \partial_tu(0,x) = g(x).
		\end{array}\right.
	\end{equation}
	In the free case (i.e., $V = 0$), it is well known that the $W^{k+1,1} \to L^\infty$ estimate of the solution operator $\cos(t\sqrt{-\Delta})$ is $O(|t|^{-1/2})$, and the $W^{k,1} \to L^\infty$ estimate of the solution operator
	$
	\frac{\sin(t\sqrt{-\Delta})}{\sqrt{-\Delta}}
	$
	is $O(|t|^{-1/2})$ for $k > \frac{1}{2}$ in dimension two. In general, one needs to use Hardy, Besov, or BMO spaces to obtain the sharp $k = \frac{n-1}{2}$ smoothness index in even dimensions; see, e.g., \cite{MSW}. Such decay bounds can also be achieved in Sobolev spaces in odd dimensions (cf. \cite{Str05}).
	
	When $V \neq 0$, Beals and Strauss \cite{B-S} first studied the $L^\infty$ decay estimates of solution operators to equation \eqref{waveequation-twoorder} in dimensions $n \geq 3$ (see also \cite{Beals, DP, GV03}). There is less work on the $W^{k,1} \to L^\infty$ dispersive estimates or \textquotedblleft{regularized}" $L^1 \to L^\infty$-type estimates for $n = 2$, where negative powers of $-\Delta + V$ are employed. In dimension two, Moulin \cite{Mou} studied the high-frequency estimates of this type. Kopylova \cite{Kop} studied local estimates and obtained a decay rate of $t^{-1}(\log t)^{-2}$ for large $t$ when zero is a regular point. Green \cite{Green14} also studied the $L^1 \to L^\infty$ dispersive estimates of the solution operator for the wave equation with a potential in dimension two, improving the time decay rate of solution operators in weighted spaces if zero is a regular point of the spectrum of $H$.
	
	Additionally, the other dimensions have also been studied. For instance, see Cardoso and Vodev \cite{CV} for dimensions $4 \leq n \leq 7$. These results generally assume that zero is regular. Erdo\u{g}an-Goldberg-Green \cite{EGG14} established low-energy $L^1 \to L^\infty$ bounds for the wave equation with a potential in four spatial dimensions and demonstrated that the loss of derivatives in the initial data for the wave equation is a high-energy phenomenon. See also \cite{Bec-Gold,  GV95, BPSZ03, BPSZ04, GV03, DP} for further decay bounds and Strichartz estimates for the wave equation with a potential, particularly in dimensions $n \geq 3$.

	\subsection{The outline of the proof}
	

	We provide a concise explanation of the key ideas underlying the proof of the main theorem. We focus exclusively on the case where zero is a regular point of \( H \), as the arguments for other cases are analogous.
	
	The starting point involves the following Stone's formulas:
	\begin{equation}\label{stoneformula-cos}
		\begin{split}
			\cos (t \sqrt{H+m^2}) P_{ac}(H) =\frac{2}{\pi i} \int_0^{\infty} \lambda^3 \cos (t \sqrt{\lambda^4+m^2})\left[R_V^{+}(\lambda^4)-R_V^{-}(\lambda^4)\right]  d \lambda,
		\end{split}
	\end{equation}
	\begin{equation}\label{stoneformula-sin}
		\begin{split}
			\frac{\sin (t \sqrt{H+m^2})}{\sqrt{H+m^2}} P_{ac}(H) =\frac{2}{\pi i} \int_0^{\infty} \frac{\lambda^3}{\sqrt{\lambda^4+m^2}} \sin (t \sqrt{\lambda^4+m^2})\left[R_V^{+}(\lambda^4)-R_V^{-}(\lambda^4)\right] d \lambda,
		\end{split}
	\end{equation}
	where $R_V^{\pm}(\lambda^4)=\left(H-\lambda^4 \mp i 0\right)^{-1}$ are the boundary operators of resolvent of $H$. Hence we need to study the expansions of the resolvent operators $R_V^{ \pm}(\lambda^4)$ by using perturbations of the following free resolvent $R_0(z)$(see, e.g., \cite{FSY}):
	\begin{equation}\label{R_0z}
		R_0(z):=\left((-\Delta)^2-z\right)^{-1}=\frac{1}{2 z^{\frac{1}{2}}}\left(R(-\Delta ; z^{\frac{1}{2}})-R(-\Delta ;-z^{\frac{1}{2}})\right), z \in \mathbb{C} \backslash[0,\infty).
	\end{equation}
	Here the resolvent $R(-\Delta; z^{\frac{1}{2}}):=\left(-\Delta-z^{\frac{1}{2}}\right)^{-1}$ with $\Im z^{\frac{1}{2}}>0$. For $\lambda \in \mathbb{R}^{+}$, we define the limiting resolvent operators by
	$$
	\begin{aligned}
		& R_0^{\pm}(\lambda):=R_0^{\pm}(\lambda \pm i 0)=\lim_{\epsilon \rightarrow 0^+}\left(\Delta^2-(\lambda \pm i \epsilon)\right)^{-1}, \\
		& R_V^{\pm}(\lambda):=R_V^{\pm}(\lambda \pm i 0)=\lim_{\epsilon \rightarrow 0^+}(H-(\lambda \pm i \epsilon))^{-1}.
	\end{aligned}
	$$
	It was well-known, by the limiting absorption principle (see, e.g., \cite{Agmon}), that $R^\pm(-\Delta; \lambda^2)$ are well-defined as bounded operators in $ B(L_s^2, L_{-s}^2)$ for any $s > \frac{1}{2}$. Consequently, $R_0^\pm(\lambda^4)$ are also well-defined on weighted spaces. This property extends to $R_V^\pm(\lambda^4)$ for $ \lambda > 0$, assuming certain decay conditions on the potential (see, e.g., \cite{FSY}).
	So by the identity \eqref{R_0z}, we obtain the following kernel of $R_0^{\pm}\left(\lambda^4\right)$ for each $\lambda>0$:
	\begin{align}
		\label{free_resolvent}
		R_0^\pm(\lambda^4)(x,y)
		=\frac{F_\pm(\lambda |x-y|)}{4\lambda^3}=\frac{F_\pm(\lambda |x|)}{4\lambda^3}-\frac{y}{4\lambda^2}\int_0^1\sgn(x-\theta y)F_\pm'(\lambda|x-\theta y|)d\theta,
	\end{align}
	where $F_\pm(s):=\pm ie^{\pm is}-e^{-s}$ and we have used the Taylor expansion  in the second equality. In particular, $R_0^\pm(\lambda^4)=O(\lambda^{-3})$ at the level of the order of $\lambda$.
	
	Note that
	$$
	\cos (t \sqrt{H+m^2})=\frac{e^{i t \sqrt{H+m^2}}+e^{-i t \sqrt{H+m^2}}}{2}\ \ \text{and}\ \ \frac{\sin (t \sqrt{H+m^2})}{\sqrt{H+m^2}}=\frac{e^{i t \sqrt{H+m^2}}-e^{-i t \sqrt{H+m^2}}}{2i\sqrt{H+m^2}}.
	$$
	Then, we deal with  \eqref{stoneformula-sin} and \eqref{stoneformula-cos} for $ m \neq 0 $ by focusing on $ \left(H + m^2\right)^{\frac{\ell}{2}} e^{-i t \sqrt{H + m^2}} $ taking $\ell = -1, 0 $, respectively. This approach is based on the following Stone's formula:
	\begin{equation}\label{stone-H-alpha}
		\begin{split}
			\left(H + m^2\right)^{\frac{\ell}{2}} e^{-i t \sqrt{H + m^2}} P_{ac}(H)
			& = \frac{2}{\pi i} \int_0^{\infty} e^{-it \sqrt{\lambda^4 + m^2}} \lambda^3 \left(\lambda^4 + m^2\right)^{\frac{\ell}{2}} \\
			& \quad \times \big[R_V^+(\lambda^4) - R_V^-(\lambda^4)\big](x,y) \, d\lambda.
		\end{split}
	\end{equation}
	For $ m = 0 $, \eqref{stoneformula-cos} similarly reduces to \eqref{stone-H-alpha} with $\ell = 0 $. However, the singularity in \eqref{stoneformula-sin} for $ m = 0 $ introduces additional complexities, making the straightforward application of \eqref{stone-H-alpha} to \eqref{stoneformula-sin} infeasible. In this case, we instead use the relationship:
	$$
	\frac{\sin (t \sqrt{H})}{\sqrt{H}} P_{ac}(H) = \frac{1}{2} \int_{-t}^{t} \cos(s \sqrt{H}) P_{ac}(H) \, ds.$$
	Therefore, the main focus for establishing  the main result is still  the precise treatment of the Stone's formula \eqref{stone-H-alpha}.
	We further decompose \eqref{stone-H-alpha} into the low energy $\{0 \leq \lambda \ll 1\}$ and the high energy $\{\lambda \gg 1\}$ two parts.
	For the high energy part, we will use the following resolvent identity:
	$$
		R_V^{\pm}(\lambda^4)=R_0^{\pm}(\lambda^4)-R_0^{ \pm}(\lambda^4)VR_0^{\pm}(\lambda^4)+R_0^{ \pm}(\lambda^4)VR_V^{\pm}(\lambda^4)VR_0^{ \pm}(\lambda^4).
	$$
	The analysis of the high energy part is comparatively straightforward, as the free resolvent $ R_0^{\pm}(\lambda^4)$  remains nonsingular for $ \lambda \geq 1 $. Consequently, we restrict our attention to the low energy part, which is given by:
	$$
		\frac{2}{\pi i} \int_0^{\infty} e^{-it \sqrt{\lambda^4 + m^2}} \lambda^3 \left(\lambda^4 + m^2\right)^{\frac{\ell}{2}} \widetilde{\chi}_1(\lambda)
		\big[R_V^{+}(\lambda^4) - R_V^{-}(\lambda^4)\big](x, y) \, d\lambda,
	$$
	where $ \widetilde{\chi}_1(\lambda):= \chi_1(\lambda^4)$ for $ \lambda > 0 $, and $ \chi_1 $  is a smooth cut-off function supported near zero.
	
	To address the low-energy contribution, it is important to derive the asymptotic expansions of the resolvent operators $ R_V^{\pm}(\lambda^4) $ near zero.
	
	Defining $ v(x) = \sqrt{|V(x)|} $, $ U(x) = \operatorname{sgn}(V(x)) $ and $ M^\pm(\lambda) = U + vR_0^\pm(\lambda^4)v $, we arrive at the following symmetric resolvent identity:
	\begin{equation}\label{id-RV}
		R_V^\pm(\lambda^4) = R_0^\pm(\lambda^4) - R_0^\pm(\lambda^4) v \big(M^\pm(\lambda)\big)^{-1} v R_0^\pm(\lambda^4).
	\end{equation}
	The asymptotic expansion of $ \big(M^\pm(\lambda)\big)^{-1} $ as $ \lambda \to 0^+ $ has been recently derived in \cite{SWY22}. In the regular case, this expansion is expressed as:
	\begin{equation}\label{M}
		\big(M^\pm(\lambda)\big)^{-1} = Q_2 A_0^0 Q_2 + \lambda Q_1 A_1^0 Q_1 + \lambda^2 \big(Q_1 A_{21}^0 Q_1 + Q_2 A_{22}^0 + A_{23}^0 Q_2\big) + \Gamma_3^0(\lambda),
	\end{equation}
	where $ A_k^0, A_{kj}^0, Q_\alpha \in \mathbb{AB}(L^2) $,  $\|\Gamma_3^0(\lambda)\|_{\mathbb{B}(L^2)} = O(\lambda^3)$, and the operators $Q_\alpha $ satisfy:
	\begin{align}\label{idea_4}
		Q_\alpha(x^k v) \equiv 0, \quad \langle x^k v, Q_\alpha f \rangle = 0,
	\end{align}
	for any $ f \in L^2$ and any integer $0 \leq k \leq \alpha - 1 $. The significance of the properties in \eqref{idea_4}
	is that, combined with  the Taylor expansion \eqref{free_resolvent} of the free resolvent, one has
	$$
		Q_\alpha v R_0^\pm(\lambda^4) = O(\lambda^{-3 + \alpha}), \quad R_0^\pm(\lambda^4) v Q_\alpha = O(\lambda^{-3 + \alpha}),
	$$
	which are less singular in $ \lambda \in (0, 1] $ compared with the free resolvent $ R_0^\pm(\lambda^4) = O(\lambda^{-3})$.

	Combining \eqref{id-RV} and \eqref{M}, we find that \eqref{stone-H-alpha} can be expressed as the sum of the free term (i.e., $V = 0 $) and six integral operators with integral kernels of the form:
	\begin{align*}
		\int_0^\infty e^{-i t \sqrt{\lambda^4 + m^2}} \lambda^{6 - \alpha - \beta} \lambda^\gamma (\lambda^4 + m^2)^{\frac{\ell}{2}}
		\widetilde{\chi}_1(\lambda) \left[R_0^\pm(\lambda^4) v Q_\alpha B Q_\beta v R_0^\pm(\lambda^4)\right](x, y) \, d\lambda,
	\end{align*}
	where $ B \in \mathbb{B}(L^2) $ varies between terms, and we set \( Q_0 = \Id \). Furthermore:
	$$
	Q_\alpha B Q_\beta \in \left\{Q_2 A_0^0 Q_2, \, Q_1 A_1^0 Q_1, \, Q_1 A_{21}^0 Q_1, \, Q_2 A_{22}^0, \, A_{23}^0 Q_2, \, \Gamma_3^0(\lambda)\right\}.
	$$
	Here, $\gamma = 0$ for $Q_\alpha B Q_\beta = Q_1 A_1^0 Q_1$  and $ \gamma = 1 $ in all other cases.
	
	Using Littlewood-Paley decomposition and Lemma \ref{estimate-integral}, these integrals can be dominated by  the following form:
	\begin{equation*}
		\sum_{N = -\infty}^{N'} 2^N\Big| \int_0^\infty e^{-i t \sqrt{2^{4N}s^4 + m^2}} e^{\pm i 2^N s \Psi(z)} s^h \varphi_0(s) \Phi\left(2^N s, m, \ell, z\right) \, ds\Big|,
	\end{equation*}
	for some fixed  $ N' \in \mathbb{Z}$. We verify that the oscillatory integral above satisfies the conditions of Lemma \ref{estimate-integral}. By applying Lemma \ref{estimate-integral} and summing over \( N \), we obtain the desired results.
	
	The paper is organized as follows. In Section \ref{sec:reso-expan}, we derive the dispersive bounds in the free case.  Section \ref{sec:low-energy} begins with a review of  the resolvent expansions near  $\lambda=0$.
	Using Stone's formula, Littlewood-Paley decomposition and oscillatory integrals, we then establish the low-energy decay estimates stated in  Theorem \ref{main_theorem}. Finally, in Section \ref{sec:high-energy}, we present the proof of Theorem \ref{main_theorem} in the high energy part.
	\bigskip


\section{The decay estimates for the free case}\label{sec:reso-expan}
In this section, we focus on deriving the decay bounds for the free case using Littlewood-Paley decomposition and oscillatory integral theory.
We first give a lemma that is crucial for estimating the integrals discussed in this paper.
\begin{lemma}\label{estimate-integral}
Let $A$ and $B$ be some subsets of $\mathbb{R}$. Suppose that $\Phi(s,m,\ell, z)$ is a function on $\mathbb{R} \times A \times B \times \mathbb{R}^n$ which is smooth in the first variable s, and satisfies for $k=0,1$,
$$
\left|\partial_s^k \Phi\left(2^N s, m, \ell,z\right)\right| \lesssim 1, \quad (s,m,\ell,z) \in[\frac{1}{4},1] \times A\times B \times \mathbb{R}^n,~ N \in \mathbb{Z}.
$$
Suppose that $\varphi_0(s)$ is a smooth function on $\mathbb{R}$ defined in \eqref{varphi_0}, $\Psi(z)$ is a nonnegative function on $\mathbb{R}^n$. For $h \in \mathbb{R}$ and $t \neq 0$, let $N_0=\left\lfloor \log_2 \frac{\Psi(z)}{|t|}\right\rfloor$ and
$$
\mathcal{K}_{N}^\pm(m,\ell,t,z):=\int_0^{\infty} e^{-i t\sqrt{ 2^{4 N}s^4+m^2}} e^{ \pm i 2^N s \Psi(z)} s^h \varphi_0(s) \Phi\left(2^N s, m,\ell, z\right) d s.
$$
Then the following estimates hold:
\begin{itemize}
\item[(i)] If $m=0$,
$$
\left|\mathcal{K}_{N}^\pm(0,\ell,t,z)\right|
\lesssim \begin{cases}\left(1+|t| 2^{2 N}\right)^{-\frac{1}{2}}, &\text{if } \left|N-N_0\right| \leq 2, \\ \left(1+|t| 2^{2 N}\right)^{-1}, & \text{if } \left|N-N_0\right|>2 .\end{cases}
$$
\item[(ii)]If $m \neq 0$ and fix $N^{\prime\prime} \in \mathbb{Z}$, then for $N \geq N^{\prime\prime}$,
$$
\begin{aligned}
\left|\mathcal{K}_{N}^\pm(m,\ell,t,z)\right|
\lesssim \begin{cases}\left\langle m\right\rangle^{\frac{1}{2}}\left(1+|t| 2^{2 N}\right)^{-\frac{1}{2}},&\text{if } \left|N-N_0\right| \leq C_{m}, \\ \left\langle m\right\rangle \left(1+|t| 2^{2 N}\right)^{-1},&\text{if } \left|N-N_0\right|>C_{m} ,\end{cases}
\end{aligned}
$$
where $C_{m}=2+\frac{1}{2}\log_2\left(1+2^{-4(N^{\prime\prime}-2)}m^2\right)$.
\vskip 0.3cm
\item[(iii)] If $m \neq 0$ and fix $N^{\prime}\in \mathbb{Z}$, then for $N \leq N^{\prime}$,
$$
\left|\mathcal{K}_{N}^\pm(m,\ell,t,z)\right|
\lesssim \left\langle m\right\rangle^{\frac{1}{2}}\left(1+|t| 2^{4 N}\right)^{-\frac{1}{2}}.
$$
\end{itemize}
\end{lemma}
Throughout this paper, $\Theta_{N_0, N}(m, t)$ always denotes the following function:
\begin{equation}\label{Theta}
\Theta_{N_0, N}(m, t):= \begin{cases}\left\langle m\right\rangle^{\frac{1}{2}}\left(1+|t| 2^{2 N}\right)^{-\frac{1}{2}}, &\text {if } \left|N-N_0\right| \leq C_{m}, \\ \left\langle m\right\rangle\left(1+|t| 2^{2 N}\right)^{-1}, &\text {if } \left|N-N_0\right|>C_{m},\end{cases}
\end{equation}
where $N_0=\left\lfloor \log _2 \frac{\Psi(z)}{|t|}\right\rfloor$, $C_{m}=2+\frac{1}{2}\log_2\left(1+2^{-4(N^{\prime\prime}-2)}m^2\right)$ and $\Psi(z)$ is a non-negative real value function on $\mathbb{R}^n$.
When $m=0$, we note that $C_0=2$.

Define $Q_m(s)=\sqrt{s^4+m^2}, s \in \mathbb{R}^+, m \neq 0$, then
$$
Q_m^{\prime}(s)=\frac{2s^3}{\left( s^4+m^2\right)^{\frac{1}{2}}}, \quad \quad Q_m^{\prime\prime}(s)=\frac{2s^6+6s^2 m^2}{\left( s^4+m^2\right)^{\frac{3}{2}}},
$$
and the following equivalence relations are used repeatedly in the proof:
\begin{itemize}
  \item[(i)] If $s \geq 1$, $Q_m^{\prime}(s) \gtrsim \frac{s}{\left\langle m\right\rangle}, ~  Q_m^{\prime\prime}(s) \gtrsim  \frac{1}{\left\langle m\right\rangle}$; \\
  \item[(ii)] If $s < 1$, $Q_m^{\prime}(s) \gtrsim \frac{s^3}{\left\langle m\right\rangle},
   ~ Q_m^{\prime\prime}(s) \gtrsim \frac{s^2}{\left\langle m\right\rangle} $.
\end{itemize}
\begin{proof}
The case $m=0$ can be followed from Lemma 2.1 in \cite{CLSY_ArXiv}. Henceforth,  we assume $m \neq 0$ and without loss of generality $t > 0$.
Let
$$U_\pm:=U_\pm\left(m,t,s,z,N\right)=t\sqrt{ 2^{4 N}s^4+m^2} \mp 2^N s \Psi(z),
$$
$$ W:=W\left(m,\ell,s,z,N\right)=s^h \varphi_0(s) \Phi\left(2^N s, m, \ell,z\right),$$
then
$$
\mathcal{K}_{N}^\pm(m,\ell,t,z)=\int_0^{\infty} e^{-iU_\pm\left(m,t,s,z,N\right)}W\left(m,\ell,s,z,N\right)ds.
$$
Since $\left|\Phi\left(2^N s, m,\ell, z\right)\right| \lesssim 1$ and $\operatorname{supp} \varphi_0(s) \subset [\frac{1}{4},1]$, then
\begin{equation}\label{bound}
\left|\mathcal{K}_{N}^\pm(m,\ell,t,z)\right| \lesssim \int_{\frac{1}{4}}^1\left|s^h \varphi_0(s)\Phi\left(2^N s, m,\ell, z\right) \right| d s \lesssim 1.
\end{equation}
\qquad \textbf{Case 1:} Let $N\geq N^{\prime\prime}$ for some fixed $N^{\prime\prime} \in \mathbb{Z}$. In this case,
$$Q_m^{\prime}(2^N s) \gtrsim \frac{2^N s}{\left\langle m\right\rangle},
~ Q_m^{\prime\prime}(2^N s) \gtrsim \frac{1}{\left\langle m \right\rangle}.
$$
We consider $\mathcal{K}_{N}^-(m,\ell,t,z)$ and $\mathcal{K}_{N}^{+}(m,\ell,t,z)$ case by case.

 We first deal with $\mathcal{K}_{N}^-(m,\ell,t,z)$. If $t 2^{2 N} < 1$, by \eqref{bound}, it is clear that $\left|\mathcal{K}_{N}^-(m,\ell,t,z)\right|$ is bounded by $\left(1+|t| 2^{2 N}\right)^{-1}$ uniformly in $m,\ell, z$. As for $t 2^{2 N} \geq 1$, since
$$
\begin{aligned}
&\partial_s U_{-}=t2^N Q_m^{\prime}(2^N s)+2^N \Psi(z) \geq t2^N Q_m^{\prime}(2^N s) \gtrsim t2^{2N}\left\langle m \right\rangle^{-1},\\
& \partial_s^2 U_{-}=t2^{2N} Q_m^{\prime\prime}(2^N s) \gtrsim \frac{t2^{2N}}{\left\langle m \right\rangle},
\end{aligned}
$$
then by integration by parts and noting that there are no boundary terms, one has
\begin{align*}
\left|\mathcal{K}_{N}^-(m,l,t,z)\right|&\lesssim \int_{\frac{1}{4}}^1 \left|\partial_s\left(\frac{ W}{\partial_s U_{-}}\right)\right|ds\nonumber\\
&\lesssim \int_{\frac{1}{4}}^1 \frac{1}{\partial_s U_{-}}ds + \int_{\frac{1}{4}}^1 \frac{\partial_s^2 U_{-}}{\left(\partial_s U_{-}\right)^2} ds \lesssim\sup_{s\in[\frac{1}{4}, 1]} \frac{1}{\partial_s U_{-}}
\lesssim \frac{\left\langle m\right\rangle}{1+|t|2^{2N}}.
\end{align*}
And by Van der Corput lemma (see, e.g., \cite[P. 334]{Stein}), we obtain that
$$
\left|\mathcal{K}_{N}^{-}(m,\ell,t,z)\right| \lesssim\left(\frac{t2^{2N}}{\left\langle m \right\rangle}\right)^{-\frac{1}{2}}\left(W(m,\ell, 1, z, N)+\int_{\frac{1}{4}}^1\left|\partial_s W(m, \ell,s, z, N)\right| d s\right).
$$
Noting that $W(m, \ell,1, z, N)=0$ and $\left|\partial_s W(m, \ell,s, z, N)\right| \lesssim 1$, we have
\begin{equation*}
\left|\mathcal{K}_{N}^{-}(m,\ell,t,z)\right| \lesssim\left(\frac{t2^{2N}}{\left\langle m \right\rangle}\right)^{-\frac{1}{2}} \lesssim\left(1+|t| 2^{2 N}\right)^{-\frac{1}{2}}\left\langle m \right\rangle^{\frac{1}{2}}.
\end{equation*}

Next we turn to the estimate of $\mathcal{K}_{N}^{+}(m,\ell,t,z)$. If $t 2^{2 N} < 1$, it is clear that $\left|\mathcal{K}_{N}^{+}(m,\ell,t,z)\right|$ is bounded by $\left(1+|t| 2^{2 N}\right)^{-1}$ uniformly in $m, \ell,z$. If $t 2^{2 N} \geq 1$, note that $$\partial_s U_{+}=t2^N Q_m^{\prime}(2^N s)-2^N \Psi(z),$$ hence $U_{+}\left(m,t,s,z,N\right)$ may have a critical point $s_0$ located at $\left[\frac{1}{4},1\right]$ and satisfing $$2^N \Psi(z)=t2^N Q_m^{\prime}(2^N s).$$\par
If $\left|N-N_0\right| \leq C_{m}$ with $C_{m}=2+\frac{1}{2}\log_2\left(1+2^{-4(N^{\prime\prime}-2)}m^2\right)$, set
$
A=2^{-4(N^{\prime\prime}-2)}m^2$, then
$$
-2-\log_2\sqrt{1+A}+N \leq \log_2 \frac{\Psi(z)}{t} \leq N+3+\log_2\sqrt{1+A}.
$$
Furthermore, $$\frac{2^{2N-2}t}{\sqrt{1+A}} \leq 2^N \Psi(z) \leq  2^{2N+3}t\sqrt{1+A}.$$
Note that
$$
t2^N Q_m^{\prime}(2^N s)=t2^N \frac{2(2^Ns)^3}{\left((2^Ns)^4+m^2\right)^{\frac{1}{2}}} \in
\left[\frac{2^{2N-2}t}{\sqrt{1+A}},2^{2N+3}t\sqrt{1+A}\right],
$$
then the critical point $s_0$ may exist. Since
$$
\left|\partial_s^2 U_{+}\right|=\left|t2^{2N} Q_m^{\prime\prime}(2^N s)\right| \gtrsim \frac{t2^{2N}}{\left\langle m \right\rangle},
$$
hence by Van der Corput lemma, we have
$$
\left|\mathcal{K}_{N}^{+}(m,\ell,t,z)\right| \lesssim\left(\frac{t2^{2N}}{\left\langle m \right\rangle}\right)^{-\frac{1}{2}} \lesssim\left(1+|t| 2^{2 N}\right)^{-\frac{1}{2}}\left\langle m \right\rangle^{\frac{1}{2}}.
$$

If $\left|N-N_0\right| > C_{m}$, then
$$
2^N \Psi(z) \leq \frac{2^{2N-2}t}{\sqrt{1+A}}\ \text{or}\  2^N \Psi(z) \geq 2^{2N+3}t\sqrt{1+A},
$$
 thus the critical point $s_0$ doesn't exist. Using integration by parts again, we obtain that $\left|\mathcal{K}_{N}^{+}(m, \ell,t, z)\right|$ is controlled by $\left(1+|t| 2^{2 N}\right)^{-1}\left\langle m \right\rangle$.

\textbf{Case 2:} Let $N \leq N^{\prime}$ for some fixed $N^{\prime} \in \mathbb{Z}$. In this case,
$$
 Q_m^{\prime\prime}(2^N s) \gtrsim \frac{(2^N s)^2}{\left\langle m \right\rangle}.
$$

If $t2^{4N} <1$, it is clear that $\left|\mathcal{K}_{N}^{\pm}(m,\ell, t, z)\right| \lesssim \left(1+|t| 2^{4 N}\right)^{-\frac{1}{2}}$. If $t2^{4N} \geq 1$, since
$$
\left|\partial_s^2 U_{\pm}\right|=\left|t2^{2N} Q_m^{\prime\prime}(2^N s)\right| \gtrsim t2^{4N}\left\langle m \right\rangle^{-1},$$
then by Van der Corput lemma,
$$
\left|\mathcal{K}_{N}^{\pm}(m,\ell,t,z)\right| \lesssim\left(|t| 2^{4 N}\left\langle m \right\rangle^{-1}\right)^{-\frac{1}{2}} \lesssim\left(1+|t| 2^{4 N}\right)^{-\frac{1}{2}}\left\langle m \right\rangle^{\frac{1}{2}}.
$$
\end{proof}
\begin{remark}\label{estimate-integral_remark}
Note that $s^{2\ell}\leq 2^{-4l}$ for $l\leq0, \frac{1}{4}\leq s \leq 1$. Therefore, according to the above proof, if the term $s^{2\ell+h}$ appears in $\mathcal{K}_{N}^\pm(m,\ell,t,z)(\ell \leq0)$ in place of $s^h$, the upper bounds of $\left|\mathcal{K}_{N}^{\pm}(m,\ell,t,z)\right|$ will be multiplied by $2^{-4l}$.
\end{remark}
\begin{lemma}\label{sum}
Let $\Theta_{N_0, N}(m,t)$ be defined by \eqref{Theta}. Then
\begin{itemize}
\item[(i)] For $-\frac{1}{2}<\ell\leq0,$
$$
\sum_{N=-\infty}^{+\infty}2^{(1+2\ell)N} \Theta_{N_0, N}(0,t) \lesssim_\ell |t|^{-\frac{1+2\ell}{2}}.
$$
\item[(ii)]
$\sum_{N=-\infty}^{+\infty}2^N\Theta_{N_0, N}(m,t) \lesssim \left\langle m\right\rangle |t|^{-\frac{1}{2}}.
$
\vskip 0.4cm
\item[(iii)]
$
\sum_{N=-\infty}^{+\infty}2^N(1+|t|  2^{4N})^{-\frac{1}{2}} \lesssim  |t|^{-\frac{1}{4}}.
$
\end{itemize}
\end{lemma}
\begin{proof}
(i) For $t \neq 0$, there exists $N_{0}^{\prime} \in \mathbb{Z}$ such that $2^{N_{0}^{\prime}} \sim |t|^{-\frac{1}{2}}$. If $-\frac{1}{2}<\ell<0,$ then
$$
\begin{aligned}
 \sum_{N=-\infty}^{+\infty}2^{(1+2\ell)N} \Theta_{N_0, N}(0,t)
& \lesssim \sum_{N=-\infty}^{+\infty} 2^{(1+2\ell)N}(1+|t|2^{2N})^{-\frac{1}{2}}\\
& \lesssim \sum_{N=-\infty}^{N_0^{\prime}} 2^{(1+2\ell)N}+ \sum_{N=N_0^{\prime}+1}^{+\infty}2^{(1+2\ell)N}(|t|2^{2N})^{-\frac{1}{2}}\lesssim_\ell |t|^{-\frac{1+2\ell}{2}}.
\end{aligned}
$$
If $\ell=0$, then
$$
\begin{aligned}
\sum_{N=-\infty}^{+\infty}2^{N} \Theta_{N_0, N}(0,t)& \lesssim \sum_{\left|N-N_0\right| \leq 2} 2^{N}(1+|t|2^{2N})^{-\frac{1}{2}}+\sum_{\left|N-N_0\right| > 2} 2^N(1+|t|  2^{2N})^{-1} \\
\label{free_sum_m_high}
& \lesssim |t|^{-\frac{1}{2}}+ \sum_{N=-\infty}^{N_0^{\prime}} 2^N+  \sum_{N=N_0^{\prime}+1}^{+\infty} 2^{N}(|t|2^{2N})^{-1} \lesssim |t|^{-\frac{1}{2}}.
\end{aligned}
$$
(ii) Similar to (i) for $\ell=0$, since $C_m \lesssim 1+\log_2 \left\langle m \right\rangle$, we obtain that
$$
\begin{aligned}
\sum_{N=-\infty}^{+\infty}2^{N} \Theta_{N_0, N}(m,t)  \lesssim C_m\left\langle m\right\rangle^{\frac{1}{2}} |t|^{-\frac{1}{2}}+ \left\langle m\right\rangle \sum_{N=-\infty}^{N_0^{\prime}} 2^N+  \left\langle m\right\rangle \sum_{N=N_0^{\prime}+1}^{+\infty} 2^{N}(|t|2^{2N})^{-1}
\lesssim \left\langle m\right\rangle |t|^{-\frac{1}{2}}.
\end{aligned}
$$
(iii) For $t \neq 0$, there exists $N_{0}^{\prime\prime} \in \mathbb{Z}$ such that $2^{N_{0}^{\prime\prime}} \sim |t|^{-\frac{1}{4}}$. Then
\begin{align*}
\sum_{N=-\infty}^{+\infty}2^N(1+|t|  2^{4N})^{-\frac{1}{2}}& = \sum_{N=-\infty}^{N_0^{\prime\prime}} 2^N(1+|t|2^{4N})^{-\frac{1}{2}}+\sum_{N=N_0^{\prime\prime}+1}^{+\infty} 2^N(1+|t|  2^{4N})^{-\frac{1}{2}}  \\
\label{free_sum_m_low}
&\lesssim   \sum_{N=-\infty}^{N_0^{\prime\prime}}2^N+
\sum_{N=N_0^{\prime\prime}+1}^{+\infty} 2^{-N}|t|^{-\frac{1}{2}}
\lesssim  |t|^{-\frac{1}{4}}.
\end{align*}
\end{proof}

 Next we establish the $L^1-L^{\infty}$ bounds of the two  free operators  $
 \cos (t \Delta) $ and $\frac{\sin (t \Delta)}{\Delta}.$
 Note that
$$
\frac{\sin(t\Delta)}{\Delta}=\frac{1}{2}\int_{-t}^{t} \cos(s\Delta)ds,
$$
then it is sufficient to consider the decay estimate of the operator $\cos (t \Delta)$.
 By Stone's formula and  Littlewood-Paley decomposition \eqref{varphi_0}, it is enough to study the following integral kernels for each $N$ and sum up, which are given by the following Proposition \ref{free_case}.
$$
\begin{aligned}
\int_0^{\infty} e^{-i t \lambda^2} \lambda^{3+2\ell} \varphi_0(2^{-N} \lambda)R_0^{\pm}(\lambda^4)(x,y) d \lambda,~-\frac{1}{2}<\ell\leq0.
\end{aligned}
$$
\begin{proposition}\label{free_case}
We have the following estimates hold:
\begin{itemize}
\item[(i)] For $-\frac{1}{2}<\ell\leq0$,
$$
\left\|(-\Delta)^\ell e^{i t \Delta} \right\|_{L^1 \rightarrow L^{\infty}} \lesssim_\ell |t|^{-\frac{1+2\ell}{2}}.
$$
\vskip 0.3cm
\item[(ii)]
$
\|\cos (t \Delta) \|_{L^1 \rightarrow L^{\infty}} \lesssim|t|^{-\frac{1}{2}}
$
\ and \
$
 \left\|\frac{\sin (t \Delta)}{\Delta} \right\|_{L^1 \rightarrow L^{\infty}} \lesssim|t|^{\frac{1}{2}}.
$
\end{itemize}
\end{proposition}
\begin{proof}
(i) For each $N \in \mathbb{Z}$, we write
$$
K_{0, N}^{ \pm}(\ell,t, x, y):=\int_0^{\infty} e^{-i t \lambda^2} \lambda^{3+2\ell} \varphi_0(2^{-N} \lambda)R_0^{\pm}(\lambda^4)(x,y) d \lambda.
$$
Note that
$$
R_0^{ \pm}(\lambda^4)(x, y)=\frac{ \pm i e^{ \pm i \lambda|x-y|}-e^{-\lambda|x-y|}}{4 \lambda^3}:=\frac{e^{ \pm i \lambda|x-y|}}{4 \lambda^3}\mathcal{F}_{0}^{\pm}\left(\lambda|x-y|\right).
$$
Utilizing the change of variable $\lambda \rightarrow 2^Ns$, we have
$$
\begin{aligned}
K_{0, N}^{ \pm}(\ell,t, x, y)=&  \frac{2^{(1+2\ell)N}}{4} \int_0^{\infty} e^{-i t 2^{2 N} s^2} s^{2\ell} \varphi_0(s) e^{ \pm i 2^N s|x-y|}\mathcal{F}_0^{ \pm}(2^N s|x-y|) ds.
\end{aligned}
$$
Let $z=(x, y)$ and
$$
\Psi(z)=|x-y|, \quad \Phi^{ \pm}(2^N s, z)=\mathcal{F}_0^{ \pm}(2^N s|x-y|).
$$
It can be verified that for any $s \in \operatorname{supp} \varphi_0(s)$ and $k=0,1,$
$$\left|\partial_s^k\Phi^{ \pm}\left(2^N s, z\right)\right| \lesssim 1,$$ then by Lemma \ref{estimate-integral} and Remark \ref{estimate-integral_remark}, we obtain that
$$
\sup _{x, y \in \mathbb{R}}\left|K_{0, N}^{ \pm}(\ell,t, x, y)\right| \lesssim 2^{-4\ell} 2^{(1+2\ell)N} \Theta_{N_0, N}(0,t).
$$
Finally, by Lemma \ref{sum} (i), for $m=0,-\frac{1}{2}<\ell\leq0$,
$$
\sum_{N=-\infty}^{+\infty}\left|K_{0, N}^{ \pm}(\ell,t, x, y)\right| \lesssim _\ell |t|^{-\frac{1+2\ell}{2}},
$$
which implies that $\left\|(-\Delta)^\ell e^{i t \Delta} \right\|_{L^1 \rightarrow L^{\infty}} \lesssim _\ell |t|^{-\frac{1+2\ell}{2}}.$

(ii)
Since
$$
\left\| e^{i t \Delta} \right\|_{L^1 \rightarrow L^{\infty}} \lesssim|t|^{-\frac{1}{2}}, \quad
\cos (t \Delta)=\frac{e^{it\Delta}+e^{-it\Delta}}{2},
$$
then
$$
\|\cos (t \Delta) \|_{L^1 \rightarrow L^{\infty}} \lesssim|t|^{-\frac{1}{2}}.
$$
Note that
$$
\frac{\sin(t\Delta)}{\Delta}=\frac{1}{2}\int_{-t}^{t} \cos(s\Delta)ds,
$$
then
$$
\left\|\frac{\sin(t\Delta)}{\Delta}\right\|_{L^1 \rightarrow L^{\infty}} \lesssim
\int_{-t}^{t} \left\|\cos(s\Delta) \right\|_{L^1 \rightarrow L^{\infty}}ds \lesssim
\int_{-t}^{t} |s|^{-\frac{1}{2}} \lesssim |t|^{\frac{1}{2}}.
$$
\end{proof}
In the following, we consider the $L^1-L^{\infty}$ bounds of free terms
$$\cos (t \sqrt{\Delta^2+m^2}) \ \text{ and } \ \frac{\sin (t \sqrt{\Delta^2+m^2})}{\sqrt{\Delta^2+m^2}},~m \neq 0,
$$
both of which can be reduced into the form
$\left(\Delta^2+m^2\right)^{\frac{\ell}{2}}e^{-i t \sqrt{\Delta^2+m^2}},~\ell=-1,0$.
\begin{proposition}\label{free_case_m}
For $m \neq 0, ~\ell\leq0$,
$$
\begin{aligned}
&\left\|\left(\Delta^2+m^2\right)^{\frac{\ell}{2}}e^{-i t \sqrt{\Delta^2+m^2}}\chi_1(\Delta^2)\right\|_{L^1 \rightarrow L^{\infty}} \lesssim m^{\ell}\left\langle  m \right\rangle^{\frac{1}{2}}|t|^{-\frac{1}{4}}, \\
&\left\|\left(\Delta^2+m^2\right)^{\frac{\ell}{2}}e^{-i t \sqrt{\Delta^2+m^2}}\chi_2(\Delta^2)\right\|_{L^1 \rightarrow L^{\infty}} \lesssim m^{\ell}\left\langle  m \right\rangle |t|^{-\frac{1}{2}}.
\end{aligned}
$$
In particular,
for $m \neq 0$,
$$
\begin{aligned}
&\left\|\cos (t \sqrt{\Delta^2+m^2}) \chi_j(\Delta^2)\right\|_{L^1 \rightarrow L^{\infty}} \lesssim
\begin{cases}
  \left\langle  m \right\rangle^{\frac{1}{2}}|t|^{-\frac{1}{4}}, & \text{if}\  j=1, \\
  \left\langle  m \right\rangle |t|^{-\frac{1}{2}}, & \text{if} \ j=2;
\end{cases} \\
& \left \|\frac{\sin (t \sqrt{\Delta^2+m^2})}{\sqrt{\Delta^2+m^2}} \chi_j(\Delta^2)\right \|_{L^1 \rightarrow L^{\infty}}\lesssim
\begin{cases}
  m^{-1}\left\langle  m \right\rangle^{\frac{1}{2}}|t|^{-\frac{1}{4}}, & \text{if} \ j=1, \\
  \big(1+\frac{1}{m}\big) |t|^{-\frac{1}{2}}, & \text{if} \ j=2.
\end{cases}
\end{aligned}
$$
\end{proposition}
\begin{proof}
 By Stone's formula, $\left(\Delta^2+m^2\right)^{\frac{\ell}{2}}e^{-i t \sqrt{\Delta^2+m^2}}\chi_j(\Delta^2)$ can be expressed as
$$
 \int_0^{\infty} e^{-i t \sqrt{\lambda^4+m^2}}\lambda^3 (\lambda^4+m^2)^{\frac{\ell}{2}}\widetilde{\chi}_j( \lambda)
\left[R_0^{+}(\lambda^4)-R_0^{-}(\lambda^4)\right](x,y)  d \lambda,~ j=1,2,
$$
where $\widetilde{\chi}_j(\lambda):=\chi_j(\lambda^4),~\lambda>0.$
By Littlewood-Paley decomposition \eqref{varphi_0}, we can decompose $\widetilde{\chi}_j$ as follows:
$$
\widetilde{\chi}_j(\lambda)=\sum_{N=-\infty}^{+\infty}\widetilde{\chi}_j(\lambda)
\varphi_0(2^{-N}\lambda),\quad \lambda>0, ~j=1,2.
$$
Thus, in order to deal with $\left(\Delta^2+m^2\right)^{\frac{\ell}{2}}e^{-i t \sqrt{\Delta^2+m^2}}\chi_j(\Delta^2)$ for $j=1,2$, it is sufficient to focus on the following integral kernels for each $N \in \mathbb{Z}$ and sum up:
$$
K_{j, N}^{ \pm}(m,\ell,t, x, y):=\int_0^{\infty} e^{-i t \sqrt{\lambda^4+m^2}} \lambda^{3}\left(\lambda^4+m^2\right)^{\frac{\ell}{2}} \widetilde{\chi}_j(\lambda)\varphi_0(2^{-N} \lambda) R_0^{ \pm}(\lambda^4)(x, y) d \lambda,~j=1,2.
$$
Since
$$
\operatorname{supp} \widetilde{\chi}_1 \subset [0,(2\lambda_0)^{\frac{1}{4}}],~ \operatorname{supp} \widetilde{\chi}_2\subset [\lambda_0^{\frac{1}{4}},+\infty)
\ \text{and}\  \operatorname{supp} \varphi_0(2^{-N}\lambda) \subset [2^{N-2},2^N],
$$
then we may take
$
N^\prime=\left \lfloor  \frac{1}{4} \log_2\lambda_0+\frac{13}{4}\right \rfloor
 ~\text{and} ~ N^{\prime\prime}=\left \lfloor\frac{1}{4} \log_2\lambda_0\right \rfloor
$
so that
$$
K_{1, N}^{ \pm}(m,\ell,t, x, y)=0 \ \text{for}\  N\geq N^\prime \quad \text{and} \quad
K_{2, N}^{ \pm}(m,\ell,t, x, y)=0 \ \text{for}\  N\leq N^{\prime\prime}.
$$

Now, we start to pay attention to $K_{j, N}^{ \pm}(m,\ell,t, x, y)$ for $j=1,2$. Note that
$$
R_0^{ \pm}(\lambda^4)(x, y)=\frac{ \pm i e^{ \pm i \lambda|x-y|}-e^{-\lambda|x-y|}}{4 \lambda^3}:=\frac{e^{ \pm i \lambda|x-y|}}{4 \lambda^3}\mathcal{F}_{0}^{\pm}\left(\lambda|x-y|\right).
$$
Utilizing the change of variable $\lambda \rightarrow 2^Ns$, $K_{j, N}^{ \pm}(m,\ell,t, x, y)$ can be written as
$$
\begin{aligned}
\frac{2^N}{4} \int_0^{\infty} e^{-i t \sqrt{2^{4 N} s^4+m^2}} \widetilde{\chi}_j(2^N s)\varphi_0(s) e^{ \pm i 2^N s|x-y|}\left(2^{4 N} s^4+m^2\right)^{\frac{\ell}{2}}\mathcal{F}_0^{ \pm}(2^N s|x-y|) ds,~j=1,2.
\end{aligned}
$$
Let $z=(x, y)$, $\Psi(z)=|x-y|$ and for $j=1,2,$
$$
\begin{aligned}
\Phi_j^{ \pm}(2^N s, m,\ell, z)=
m^{-\ell}\left(2^{4 N} s^4+m^2\right)^{\frac{\ell}{2}}\widetilde{\chi}_j(2^N s)\mathcal{F}_0^{ \pm}(2^N s|x-y|), ~ m \neq 0,\ell\leq 0.
\end{aligned}
$$
It can be verified that for any $s \in \operatorname{supp} \varphi_0(s)$ and $k=0,1,$
$$\left|\partial_s^k \Phi_j^{ \pm}\left(2^N s, m,\ell, z\right)\right| \lesssim 1,~j=1,2,$$ then by Lemma \ref{estimate-integral}, we obtain that
for $m \neq 0, \ell\leq0$,
$$
\begin{aligned}
& \sup _{x, y \in \mathbb{R}}\left|K_{1, N}^{ \pm}(m,\ell,t, x, y)\right| \lesssim 2^N m^{\ell} \left\langle  m \right\rangle^{\frac{1}{2}} \left(1+|t| 2^{4 N}\right)^{-\frac{1}{2}}, \quad N \leq N^\prime, \\
& \sup _{x, y \in \mathbb{R}}\left|K_{2, N}^{ \pm}(m,\ell,t, x, y)\right| \lesssim  2^N m^{\ell} \Theta_{N_0, N}(m, t),\quad N \geq N^{\prime\prime}.
\end{aligned}
$$
Finally, by Lemma \ref{sum} (ii) and (iii), for $m\neq0,\ell\leq0$,
$$
\begin{aligned}
 \sum_{N=-\infty}^{+\infty} \left|K_{1, N}^{ \pm}(m,\ell,t, x, y)\right| \lesssim
m^{\ell}\left\langle  m \right\rangle^{\frac{1}{2}}|t|^{-\frac{1}{4}}, \quad
 \sum_{N=-\infty}^{+\infty} \left|K_{2, N}^{ \pm}(m,\ell,t, x, y)\right| \lesssim m^{\ell} \left\langle  m \right\rangle |t|^{-\frac{1}{2}}.
\end{aligned}
$$

Thus we complete the proof.
\end{proof}
\section{Low energy decay estimates}\label{sec:low-energy}
\subsection{Asymptotic expansions of resolvent near zero}\label{subsec:Q}
This subsection is mainly devoted to recalling the asymptotic
expansions of $(M^{\pm}(\lambda))^{-1}$  at low energy $\lambda\to 0^+$, see \cite{SWY22}.

Let $T_0:=U+v G_0 v$ with $G_0=(\Delta^2)^{-1}$ and $P=\|V\|_{L^1(\mathbb{R})}^{-1} v\langle v, \cdot\rangle$ be the orthogonal projection onto the span of $v$, i.e., $P L^2(\mathbb{R})=\operatorname{span}\{v\}$. Let $ Q_1:=\Id-P$ and define the subspaces $Q_{2}L^2,Q_{2}^{0}L^2,Q_{3}L^2$ of $L^2$ by
\begin{align*}
f\in Q_{2}L^2\ &\Longleftrightarrow\ f\in \mathrm{\mathop{span}}\{v,xv\}^\perp;\\
f\in Q_{2}^{0}L^2\ &\Longleftrightarrow\ f\in \mathrm{\mathop{span}}\{v,xv\}^\perp\ \text{and}\ T_0f\in \mathrm{\mathop{span}}\{v,xv\};\\
f\in Q_{3}L^2\ &\Longleftrightarrow\ f\in \mathrm{\mathop{span}}\{v,xv,x^2v\}^\perp\ \text{and}\ T_0f\in \mathrm{\mathop{span}}\{v\}.
\end{align*}
Note that $Q_{3}L^2\subset Q_{2}^{0}L^2\subset Q_{2}L^2\subset Q_{1}L^2$. Let $Q_{\alpha}$ and $Q_{2}^0$ be the orthogonal projection onto $Q_{\alpha}L^2$ and $Q_{2}^0L^2$, respectively. By definition, $Q_1,Q_2,Q_2^0,Q_{3}$ satisfy
\begin{equation}
\label{Q}
Q_{\alpha}(x^k v)=0,\quad \<x^kv,Q_{\alpha}f\>=0,\quad Q_{2}^0(x^k v)=0,\quad \<x^k v,Q_{2}^0f\>=0,
\end{equation}
here $k=0$ for $Q_1$, $k=0,1$ for $Q_{2},Q_2^0$ and $k=0,1,2$ for $Q_{3}$ .

These projection spaces above can be used to characterized the zero energy resonance types of $H$.
\begin{theorem}[{\cite[Theorem 1.7]{SWY22}}]
Assume that $H=\Delta^2+V$ with $\left|V(x)\right| \lesssim \left \langle  x\right \rangle^{-\mu}$ with some $\mu>0$ satisfying the condition \eqref{condition}. Then the following statements hold:
\begin{itemize}
\item zero is a regular point of $H$ if and only if $Q_2^0L^2(\mathbb{R})=\{0\};$
\item zero is the first kind resonance of $H$ if and only if  $Q_2^0L^2(\mathbb{R}) \neq \{0\}$ and $Q_3L^2(\mathbb{R})=\{0\};$
\item zero is the second kind resonance of $H$ if and only if  $Q_3L^2(\mathbb{R}) \neq \{0\}.$
\end{itemize}
\end{theorem}
Finally we recall the expansions of $\left(M^\pm(\lambda)\right)^{-1}$.
\begin{theorem}[{\cite[Theorem 1.8]{SWY22}},{\cite[Lemma 2.2]{MWY22}}]\label{lem_M}
Assume that $|V(x)| \lesssim \left \langle  x\right \rangle^{-\mu}$ with some $\mu>0$. Then there exists $0<\lambda_0 \ll 1$ such that $\left(M^{\pm}(\lambda)\right)^{-1}$ satisfies the following asymptotic expansions on $L^2(\R)$ for $0<\lambda\le \lambda_0$:
\begin{itemize}
\item[(i)] If zero is a regular point of $H$ and $\mu>13$, then
\begin{equation}\label{lem_M_01}
\left(M^{\pm}(\lambda)\right)^{-1}= Q_{2}A_{0}^0Q_{2}+\lambda Q_1A_{1}^0Q_1+\lambda^2\left(Q_1A_{21}^0Q_1+Q_{2}A_{22}^0+A_{23}^0Q_{2}\right)+\Gamma_3^0(\lambda);
\end{equation}
\item[(ii)] If zero is the first kind resonance of $H$ and $\mu>17$, then
\begin{equation}\label{lem_M_02}
\begin{split}
\left(M^{\pm}(\lambda)\right)^{-1}= &\lambda^{-1}Q_{2}^{0}A_{-1}^1Q_{2}^{0}+Q_{2}A_{01}^1Q_1+Q_1A_{02}^1Q_{2}+\lambda\left(Q_1A_{11}^1Q_1
\right.\\
&\left.\quad +Q_{2}A_{12}^1+A_{13}^1Q_{2}\right)
 +\lambda^2\left(Q_1A_{21}^1+A_{22}^1Q_1\right)+\Gamma_3^1(\lambda);
\end{split}
\end{equation}
\item[(iii)] If zero is the second kind resonance of $H$ and $\mu>25$, then
\begin{equation}\label{lem_M_03}
\begin{split}
\left(M^{\pm}(\lambda)\right)^{-1}= & \lambda^{-3}Q_{3}A_{-3}^2Q_{3}+\lambda^{-2}\left(Q_{3}A_{-21}^2Q_{2}+Q_{2}A_{-22}^2Q_{3}\right)
+\lambda^{-1}\left(Q_{2}A_{-11}^2Q_{2}\right.\\
&\left.+Q_{3}A_{-12}^2Q_1+Q_1A_{-13}^2Q_{3}\right) +Q_{2}A^2_{01}Q_1+Q_1A^2_{02}Q_{2}+Q_{3}A^2_{03}+A^2_{04}Q_{3}\\
&+\lambda\left(Q_1A_{11}^2Q_1+Q_{2}A^2_{12}+A^2_{13}Q_{2}\right)+\lambda^2\left(Q_1A^2_{21}+A^2_{22}Q_1\right)
+\Gamma_3^2(\lambda).
\end{split}
\end{equation}
\end{itemize}
Here in the statements (i)-(iii) above,
	\begin{itemize}
\item $A_{i}^k$ and $A_{ij}^k$ are $\lambda$-independent bounded operators in $L^2(\mathbb{R})$ such that all of the terms in \eqref{lem_M_01}-\eqref{lem_M_03} are absolutely bounded operators, where $k,i,j$ indicate the types of resonance the degree of $\lambda$ and the order of operators with the same power of $\lambda$, respectively.
\item $\Gamma_3^k(\lambda)$ is $\lambda$ dependent operator which satisfies
\begin{equation}\label{lem_Gambda}
\left\|\Gamma_3^k(\lambda)\right\|_{L^2 \rightarrow L^2}+\lambda\left\|\partial_\lambda\left(\Gamma_3^k(\lambda)\right)\right\|_{L^2 \rightarrow L^2} \lesssim \lambda^3, \quad k=0,1,2.
\end{equation}
\item The operators $A_{i}^k$, $A_{ij}^k$ and $\Gamma_3^k(\lambda)$ all depend on the sign $\pm$, but for simplity, we ignore the sign $\pm$.
\end{itemize}
\end{theorem}
\begin{remark}
For convenience, we adopt the notations $Q_1,Q_{2},Q_2^0,Q_{3}$ introduced in \cite{MWY22}, rather than $Q,S_0,S_1,S_2$ used in \cite{SWY22}.
\end{remark}
\subsection{Low energy decay estimate}\label{low_estimate}

In this subsection, we are devoted to establishing the low energy decay bounds for Theorem \ref{main_theorem}, that is, the following Theorem \ref{main_theorem_low}.

By using Stone's formula,
\begin{align}
\nonumber
 (H&+m^2)^{\frac{\ell}{2}}e^{-i t \sqrt{H+m^2}}P_{ac}(H)\chi_1(H)    \\
 \label{stone_low_1}
&= \frac{2}{\pi i} \int_0^{\infty} e^{-it \sqrt{\lambda^4+m^2}}\lambda^3 \left(\lambda^4+m^2\right)^{\frac{\ell}{2}}\widetilde{\chi}_1(\lambda) [R_V^{+}(\lambda^4)-R_V^{-}(\lambda^4)](x,y)  d \lambda,
\end{align}
where $\widetilde{\chi}_1(\lambda)=\chi_1(\lambda^4)( \lambda>0)$ so that $\operatorname{supp}\widetilde{\chi}_1 \subset [0,(2\lambda_0)^{\frac{1}{4}}]$.
And note that
$$
\begin{aligned}
\frac{\sin (t \sqrt{H})}{\sqrt{H}} P_{a c}(H)\chi_1(H) =\frac{1}{2}\int_{-t}^{t}\cos(s\sqrt{H})P_{a c}(H)\chi_1(H)ds.
\end{aligned}
$$
Then it is sufficient to consider \eqref{stone_low_1} when $m=0,\ell=0$ and $m\neq0,\ell=-1,0.$
\begin{theorem}\label{main_theorem_low}
Let $H$ and $V$ satisfy the same conditions as in Theorem \ref{main_theorem}, then
$$
\begin{aligned}
& \left\|e^{-i t \sqrt {H}} P_{a c}(H) \chi_1(H)\right\|_{L^1 \rightarrow L^{\infty}} \lesssim|t|^{-\frac{1}{2}}, \\
& \left \|\frac{\sin (t \sqrt{H})}{\sqrt{H}} P_{a c}(H) \chi_1(H)\right \|_{L^1 \rightarrow L^{\infty}} \lesssim|t|^{\frac{1}{2}},
\end{aligned}
$$
 and for $m \neq 0$,
$$
\begin{aligned}
&\left\|\cos (t \sqrt{H+m^2}) P_{a c}(H) \chi_1(H)\right\|_{L^1 \rightarrow L^{\infty}} \lesssim \left\langle  m \right\rangle^{\frac{1}{2}}|t|^{-\frac{1}{4}}, \\
& \left \|\frac{\sin (t \sqrt{H+m^2})}{\sqrt{H+m^2}} P_{a c}(H) \chi_1(H)\right \|_{L^1 \rightarrow L^{\infty}}
\lesssim m^{-1} \left\langle  m \right\rangle^{\frac{1}{2}} |t|^{-\frac{1}{4}}.
\end{aligned}
$$
\end{theorem}

 We first give the following elementary lemma which plays an important role in making use of cancellations of operators $Q_{2}^0$ and $Q_{\alpha}, \alpha=1,2,3$.
\begin{lemma}[{\cite[Lemma 2.5]{SWY22}}]\label{taylor}
 Let $\lambda>0$ and $\operatorname{sgn}(x)$ be the sign function of $x$ on $\mathbb{R}$. Then
\begin{itemize}
\item[(i)] If $F(p) \in C^1(\mathbb{R})$, then for any $x, y \in \mathbb{R}$, we have
$$
F(\lambda|x-y|)=F(\lambda|x|)-\lambda y \int_0^1 \operatorname{sgn}(x-\theta y) F^{\prime}(\lambda|x-\theta y|) \mathrm{d} \theta.
$$
\item[(ii)] If $F(p) \in C^2(\mathbb{R})$ and $F^{\prime}(0)=0$, then for any $x, y \in \mathbb{R}$, we have
$$
F(\lambda|x-y|)=F(\lambda|x|)-\lambda y \operatorname{sgn}(x) F^{\prime}(\lambda|x|)
+\lambda^2 y^2 \int_0^1(1-\theta) F^{\prime \prime}(\lambda|x-\theta y|) \mathrm{d} \theta.
$$
\item[(iii)] If $F(p) \in C^3(\mathbb{R})$ and $F^{\prime}(0)=F^{\prime \prime}(0)=0$, then for any $x, y \in \mathbb{R}$, we have
$$
\begin{aligned}
F(\lambda|x-y|)= & F(\lambda|x|)-\lambda y \operatorname{sgn}(x) F^{\prime}(\lambda|x|)+\frac{\lambda^2 y^2}{2!} F^{\prime \prime}(\lambda|x|) \\
& -\frac{\lambda^3 y^3}{2!} \int_0^1(1-\theta)^2(\operatorname{sgn}(x-\theta y))^3 F^{(3)}(\lambda|x-\theta y|) \mathrm{d} \theta.
\end{aligned}
$$
\end{itemize}
\end{lemma}
We will mainly use this lemma for
$
F_\pm(s)=\pm ie^{\pm i s}-e^{-s}
$.
Combined with \eqref{free_resolvent} and \eqref{Q}, Lemma \ref{taylor} implies the following formulas.

\begin{lemma}[{\cite[Lemma 2.5]{MWY22}}]
\label{lemma_projection}
Let $Q_1,Q_{2},Q_{2}^{0},Q_{3}$ be as in Subsection \ref{subsec:Q}, $\alpha=0,1,2,3$. Then
$$
\begin{aligned}
&[Q_{\alpha} vR_0^\pm(\lambda^4)f](x)\\
&=\frac{(-1)^\alpha \lambda^{-3+\alpha}}{4\cdot (\alpha-1)!} Q_{\alpha}\left(x^\alpha v\int \int_0^1(1-\theta)^{\alpha-1}\left(\sgn(y-\theta x)\right)^\alpha  F_\pm^{(\alpha)}(\lambda|y-\theta x|)f(y)d\theta dy\right),\\
&[R_0^\pm(\lambda^4)vQ_{\alpha}f](x)\\
&=\frac{(-1)^\alpha \lambda^{-3+\alpha}}{4\cdot (\alpha-1)!}\int \int_0^1(1-\theta)^{\alpha-1}\!\left(\sgn(x-\theta y)\right)^\alpha  F_\pm^{(\alpha)}(\lambda|x-\theta y|)y^\alpha v(y)(Q_{\alpha}f)(y)d\theta dy,
\end{aligned}
$$
here $F_\pm^{(\alpha)}$ denotes the $\alpha$-th derivative of $F_\pm$, and for simplicity we have used the convention that $(\sgn x)^2\equiv1$ for all $x\in \R$. Moreover, these estimates for $\alpha=2$ also hold with $Q_2$ replaced by $Q_2^0$.
\end{lemma}
\begin{remark}
Note that the subscript $\alpha $ of $Q_{\alpha}$ coincides with the order of differentiation for $F_\pm$. This is why we use the notations $Q_1,Q_{2},Q_{2}^{0},Q_{3}$ instead of the original ones.
\end{remark}
For convenience, in the following we provide a brief proof of Lemma \ref{lemma_projection}, and refer to \cite{MWY22} for more details.
\begin{proof}[Proof of Lemma \ref{lemma_projection}]
Since $F_\pm'(0)=0$, we can apply Lemma \ref{taylor} to $F_\pm$, which gives
\begin{align*}
R_0^\pm(\lambda^4)(x,y)
&=\frac{F_\pm(\lambda |x|)}{4\lambda^3}-\frac{y}{4\lambda^2}\int_0^1\sgn(x-\theta y)F_\pm'(\lambda|x-\theta y|)d\theta\\
&=\frac{F_\pm(\lambda |x|)}{4\lambda^3}-\frac{y \sgn(x)F_\pm'(\lambda|x|)}{4\lambda^2}+\frac{y^2}{4\lambda}\int_0^1(1-\theta)F_\pm''(\lambda|x-\theta y|)d\theta.
\end{align*}
The cases $\alpha =1,2$ follow from this formula and \eqref{Q}. For the case $\alpha =3$, we write
$$4\lambda^3R_0^\pm(\lambda^4)(x,y)=F_\pm(\lambda|x-y|)=\widetilde F_\pm(\lambda|x-y|)-\frac{1\pm i}{2}\lambda^2|x-y|^2,$$ where $
\widetilde F_\pm(s)=F_\pm(s)+\frac{1\pm i}{2}s^2
$. Then we can write
\begin{align*}
Q_{3}vR_0^\pm(\lambda^4)f=\frac{1}{4\lambda^3}Q_{3}\left\{v\int\left(\widetilde F_\pm(\lambda|x-y|)-\frac{1\pm i}{2}\lambda^2|x-y|^2\right)f(y)dy\right\}.
\end{align*}
For the first term of the right hand side, since $\widetilde F_\pm'(0)=\widetilde F_\pm''(0)=0$ and $F_\pm^{(3)}\equiv\widetilde F_\pm^{(3)}$, we can apply Lemma \ref{taylor} (iii) and \eqref{Q} to compute
\begin{align*}
&Q_{3}\left(v \int \widetilde F_\pm(\lambda|x-y|)f(y)dy\right)\\
&=-\frac{\lambda^3}{2}Q_{3}\left(x^3v\int \int_0^1(1-\theta)^2\sgn(y-\theta x)F_\pm^{(3)}(\lambda|y-\theta x|) f(y)d\theta dy\right),
\end{align*}
while the second part involving $|x-y|^2$ vanishes identically by virtue of \eqref{Q}. The proof for $R_0^\pm(\lambda^4)vQ_{3}f$ is analogous.
\end{proof}
\subsection{Regular case}
We first consider the regular case. Substituting the expansion \eqref{lem_M_01} into \eqref{id-RV}, we obtain
$$
\begin{aligned}
R_V^{ \pm}(\lambda^4)= &R_0^\pm(\lambda^{4})
 -R_0^\pm(\lambda^{4})v\Big\{Q_{2}A_{0}^0Q_{2}+\lambda Q_1A_{1}^0Q_1+\lambda^2\left(Q_1A_{21}^0Q_1 \right. \\
 & \left.+Q_{2}A_{22}^0+A_{23}^0Q_{2}\right) +\Gamma_3^0(\lambda)\Big\}vR_0^\pm(\lambda^{4}).
\end{aligned}
$$
Inserting this formula into \eqref{stone_low_1}, and combining with Propositions \ref{free_case} and \ref{free_case_m}, in order to prove Theorem \ref{main_theorem_low}, it suffices to study the following integral kernels
\begin{align}
\label{formA_1}
 \int_0^{\infty} e^{-it \sqrt{\lambda^4+m^2}}\lambda^3 \left(\lambda^4+m^2\right)^{\frac{\ell}{2}}\widetilde{\chi}_1(\lambda) \left[R_0^\pm(\lambda^{4})vAvR_0^\pm(\lambda^{4})\right](x,y)  d \lambda,
\end{align}
with
$$
A=Q_{2}A_{0}^0Q_{2},~\lambda Q_1A_{1}^0Q_1,~\lambda^2 Q_1A_{21}^0Q_1,~\lambda^2Q_{2}A_{22}^0,~\lambda^2A_{23}^0Q_{2},~\Gamma_3^0(\lambda).
$$
We first deal with the cases with the first five operators,
while the last one is deferred to Proposition \ref{prop_Gamma}.

Define $Q_0:=\Id$. It can be confirmed that the corresponding integral kernels for the five operators can be written in the following form, denoted by $K^\pm(m,\ell,t, x, y)$:
\begin{align}
\label{K_form}
K^\pm=\int_0^\infty  e^{-i t \sqrt{\lambda^4+m^2}}\lambda^{6-\alpha-\beta}\lambda^\gamma (\lambda^4+m^2)^{\frac{\ell}{2}}
\widetilde{\chi}_1(\lambda)\left[R_0^\pm(\lambda^4)vQ_\alpha BQ_\beta vR_0^\pm(\lambda^4)\right](x,y)d\lambda,
\end{align}
with $Q_\alpha BQ_\beta \in \mathbb{AB}(L^2)$, $\gamma=0$ for $Q_\alpha BQ_\beta =Q_1A_{1}^0Q_1$ and $\gamma=1$ in other cases.

 Note that, since $|v(x)|\lesssim \<x\>^{-\mu/2}$ with $\mu>13$ by the assumption on $V$, $\<x\>^kv Q_\alpha BQ_\beta v\<x\>^k$ is an absolutely bounded integral operator for any $k\le 6$ at least, satisfying
\begin{align}
\label{vBv}
\int_{\R^2}\<x\>^k|(v Q_\alpha BQ_\beta v)(x,y)|\<y\>^kdxdy\lesssim \norm{\<x\>^{2k}V}_{L^1}<\infty,
\end{align}
where, denoting the integral kernel of $Q_\alpha BQ_\beta$ by $(Q_\alpha BQ_\beta)(x,y)$, we use the notation $$(vQ_\alpha BQ_\beta v)(x,y)=v(x)(Q_\alpha BQ_\beta)(x,y)v(y).$$
\begin{remark}
Note that the two projections $Q_{2},Q_{2}^{0}$ will play completely the same role in the following arguments. Hence, in what follows, we do not  distinguish them and use the same notation $Q_{2}$ to denote these operators $Q_{2},Q_{2}^{0}$.
\end{remark}
\begin{proposition}\label{prop_K}
Let $K^\pm(m,\ell,t, x, y)$ be defined by \eqref{K_form}. Then
$$
\sup _{x, y \in \mathbb{R}}\left| K^{\pm}(m,\ell,t, x, y)\right| \lesssim
\begin{cases}
  C_\ell |t|^{-\frac{1+2\ell}{2}}, & \text{if }  m=0,-\frac{1}{2}<\ell\leq0, \\
  m^{\ell} \left\langle  m \right\rangle^{\frac{1}{2}}|t|^{-\frac{1}{4}}, & \text{if } m\neq 0,\ell\leq0.
\end{cases}
$$
\end{proposition}
\begin{proof}
We consider three cases (i) $\alpha,\beta\neq0$, (ii) $\beta=0$ and (iii) $\alpha=0$ separately.

{\it Case (i)}. We suppose $\alpha,\beta\neq0$. Using Lemma \ref{lemma_projection}, \
\begin{equation}\label{r}
\begin{split}
\nonumber
&\lambda^{6-\alpha-\beta}[R_0^\pm(\lambda^4)vQ_\alpha BQ_\beta vR_0^\pm(\lambda^4) ](x,y)\\
\nonumber
&=C_\alpha C_\beta \int_{[0,1]^2\times\R^2}(1-\theta_1)^{\alpha-1}(1-\theta_2)^{\beta-1}(\sgn (X_1))^\alpha
(\sgn (Y_2))^\beta F_\pm^{(\alpha)}(\lambda|X_1|)F_\pm^{(\beta)}(\lambda|Y_2|) \\
&\quad u_1^\alpha  (vQ_\alpha  BQ_\beta  v)(u_1,u_2) u_2^\beta d\theta_1 d\theta_2 du_1 du_2\\
&=\int_{\R^2\times[0,1]^2}M_{\alpha\beta}(X_1,Y_2,\Omega)F_\pm^{(\alpha)}(\lambda|X_1|)
F_\pm^{(\beta)}(\lambda|Y_2|) d\Omega,
\end{split}
\end{equation}
where $C_\alpha=(-1)^\alpha/4\cdot(\alpha-1)!, \Omega=(u_1,u_2,\theta_1,\theta_2), X_1=x-\theta_1u_1,Y_2=y-\theta_2u_2$ and $M_{\alpha\beta}(x,y,\Omega)$ is defined by
$$
\frac{(-1)^{\alpha+\beta}(1-\theta_1)^{\alpha-1}(1-\theta_2)^{\beta-1}(\sgn x)^\alpha (\sgn y)^\beta u_1^\alpha u_2^\beta (vQ_\alpha  BQ_\beta  v)(u_1,u_2)}{16(\alpha-1)!(\beta-1)!} .
$$
Then $K^\pm$ can  be rewriten  as
$$
\begin{aligned}
\int_{\R^2\times[0,1]^2}M_{\alpha\beta}(X_1,Y_2,\Omega)
\Big(\int_0^\infty  e^{-i t \sqrt{\lambda^4+m^2}}e^{\pm i\lambda(|X_1|+|Y_2|)} \lambda^{\gamma } \widetilde{\chi}_1(\lambda)f_{\alpha \beta}^{\pm}(\lambda,m,\ell,X_1,Y_2) d\lambda \Big)d\Omega,
\end{aligned}
$$
where
$
f_{\alpha\beta}^{\pm}(\lambda,m,\ell,X_1,Y_2):=\left(\lambda^4+m^2\right)^{\frac{\ell}{2}}
\mathcal{F}_{\alpha}^{\pm}\left(\lambda|X_1|\right)\mathcal{F}_{\beta}^{\pm}\left(\lambda|Y_2|\right), \ell\leq0$
and
$
\mathcal{F}_{\alpha}^{ \pm}\left(s\right)=e^{\mp i s}F_\pm^{(\alpha)}(s).
$

By Littlewood-Paley decomposition \eqref{varphi_0}, we can decompose $\widetilde{\chi}_1$ as follows:
\begin{align}\label{decomposition_low}
\widetilde{\chi}_1(\lambda)=\sum_{N=-\infty}^{N^\prime}\widetilde{\chi}_1(\lambda)
\varphi_0(2^{-N}\lambda),\quad \lambda>0,
\end{align}
where we may take $N^\prime=\left \lfloor  \frac{1}{4} \log_2\lambda_0+\frac{13}{4}\right \rfloor$ since $\operatorname{supp} \widetilde{\chi}_1 \subset [0,(2\lambda_0)^{\frac{1}{4}}], \operatorname{supp} \varphi_0(2^{-N}\lambda)\subset [2^{N-2},2^N]$.
Then it suffices to deal with the following integral kernels $K_N^\pm (m,\ell,t,x,y)$ ($K_N^\pm$ for short) for each $N \leq N^\prime$ and sum up:
$$
\begin{aligned}
\int_{\R^2\times[0,1]^2}M_{\alpha\beta}(X_1,Y_2,\Omega) \Big(\int_0^\infty  e^{-i t \sqrt{\lambda^4+m^2}}e^{\pm i\lambda(|X_1|+|Y_2|)}
\lambda^{\gamma } \widetilde{\chi}_1(\lambda) \varphi_0(2^{-N} \lambda) f_{\alpha \beta}^{\pm} d\lambda \Big)d\Omega.
\end{aligned}
$$
Utilizing the change of variable $\lambda \rightarrow 2^Ns$, we rewrite the integral inside the parentheses with respect to $\lambda$ as follows:
\begin{equation}\label{E_1}
2^{(1+\gamma )N}\int_0^\infty  e^{-i t \sqrt{2^{4 N}s^4+m^2}}e^{\pm i2^N s(|X_1|+|Y_2|)}
s^{\gamma }\widetilde{\chi}_1(2^N s)\varphi_0(s)f_{\alpha \beta}^{\pm}(2^N s,m,\ell,X_1,Y_2) ds.
\end{equation}
Note that $M_{\alpha\beta}$ and $f_{\alpha \beta}^{\pm}$ have following several properties:
\begin{itemize}
\item By \eqref{vBv}, $ M_{\alpha\beta}(x,y,\Omega) \in L^1(\R^2\times[0,1]^2;L^\infty(\R^2_{x,y}))$ and
\begin{equation}
\label{proposition_regular}
\int_{\R^2\times[0,1]^2}\sup_{x,y\in \R^2} |M_{\alpha\beta}(x,y,\Omega)| d\Omega\lesssim \norm{\<x\>^{2(\alpha+\beta)}V}_{L^1}.
\end{equation}
\item For any $s \in \operatorname{supp} \varphi_0(s)$ and $k=0,1$,
      \begin{equation}\label{m}
      \left|\partial_s^k \left(m^{-\ell} \left(2^{4N}s^4+m^2\right)^{\frac{\ell}{2}}\right)\right| \lesssim 1, ~ m \neq 0, \ell\leq0,
      \end{equation}
      \begin{equation}\label{e}
      \left|\partial_s^k \mathcal{F}_{\alpha}^{\pm}\left(2^{N}s|X_1|\right)\right| \lesssim 1 ,
      \quad \left|\partial_s^k \mathcal{F}_{\beta}^{\pm}\left(2^{N}s|Y_2|\right)\right| \lesssim 1.
      \end{equation}
\item By \eqref{m} and \eqref{e}, for any $s \in \operatorname{supp} \varphi_0(s)$ and $k=0,1$,
      \begin{equation}\label{f}
      \left|\partial_s^k f_{\alpha \beta}^{\pm}(2^N s,m,\ell,X_1,Y_2)\right| \lesssim
      \begin{cases}
        (2^N s)^{2\ell}, & \text{if } m=0,-\frac{1}{2}<\ell\leq0, \\
        m^{\ell}, & \text{if } m \neq 0,\ell\leq0 .
      \end{cases}
      \end{equation}
\end{itemize}
Let $z=\left(x, y, y_1, y_2, \theta_1, \theta_2\right)$ and $\Psi(z)=|X_1|+|Y_2|,$
$$
\begin{aligned}
\Phi^{ \pm}(2^N s, m,\ell, z)=
\begin{cases}
  (2^N s)^{-2\ell}\widetilde{\chi}_1(2^N s)f_{\alpha \beta}^{\pm}(2^N s,m,\ell,X_1,Y_2), & \text{if } m=0,-\frac{1}{2}<\ell\leq0, \\
  m^{-\ell}\widetilde{\chi}_1(2^N s)f_{\alpha \beta}^{\pm}(2^N s,m,\ell,X_1,Y_2), & \text{if } m \neq 0,\ell\leq0,
\end{cases}
\end{aligned}
$$
then by Lemma \ref{estimate-integral} and Remark \ref{estimate-integral_remark}, we obtain that the integral \eqref{E_1} is controlled by
$$
\begin{aligned}
\begin{cases}
  2^{-4\ell} 2^{(1+2\ell)N}\Theta_{N_0, N}(0,t), & \text{if } m=0,-\frac{1}{2}<\ell\leq 0, \\
  m^{\ell} \left\langle  m \right\rangle^{\frac{1}{2}} 2^{N}\left(1+|t|2^{4N}\right)^{-\frac{1}{2}}, & \text{if } m \neq 0,\ell\leq0.
\end{cases}
\end{aligned}
$$
Furthermore, by \eqref{proposition_regular},
$$
\begin{aligned}
\left|K_N^\pm (0,0,t,x,y)\right| & \lesssim  2^{-4\ell} 2^{(1+2\ell)N}\Theta_{N_0, N}(0,t)\int_{\R^2\times[0,1]^2}\sup_{x,y\in \R^2} |M_{\alpha\beta}(x,y,\Omega)| d\Omega \\
& \lesssim 2^{-4\ell} 2^{(1+2\ell)N}\Theta_{N_0, N}(0,t).
\end{aligned}
$$
Similarly, we can check that $\left|K_N^\pm (m,\ell,t,x,y)\right|$ is controlled by $$m^{\ell} \left\langle  m \right\rangle^{\frac{1}{2}} 2^{N}\left(1+|t|2^{4N}\right)^{-\frac{1}{2}},~m \neq 0,\ell\leq0.$$
Finally, by Lemma \ref{sum} (i) and (iii), we immediately get the desired conclusions for the case $\alpha,\beta\neq0$.

{\it Case (ii)}. Let $\beta=0$, $\alpha\neq0$. As in the case (i), $K^\pm$ can be written in the form
$$
\int_{\R^2\times[0,1]}M_{\alpha0}(X_1,\Omega_1)\Big(\int_0^\infty  e^{-i t \sqrt{\lambda^4+m^2}}e^{\pm i\lambda(|X_1|+|y-u_2|)} \lambda^{\gamma } \widetilde{\chi}_1(\lambda) f_{\alpha 0}^\pm(\lambda,m,\ell,X_1,y-u_2) d\lambda \Big)d\Omega_1,$$
where $\Omega_1=(u_1,u_2,\theta_1)$, $X_1=x-\theta_1u_1$ and
$$
\begin{aligned}
&f_{\alpha 0}^\pm(\lambda,m,\ell,X_1,y-u_2)=\left(\lambda^4+m^2\right)^{\frac{\ell}{2}}\mathcal{F}_{\alpha}^{ \pm}\left(\lambda|X_1|\right)\mathcal{F}_{0}^{\pm}\left(\lambda|y-u_2|\right),~\mathcal{F}_{0}^{ \pm}\left(s\right)=e^{\mp i s}F_\pm(s), \\
&M_{\alpha0}(x,\Omega_1)=\frac{(-1)^\alpha }{16(\alpha-1)!}(1-\theta_1)^{\alpha-1}(\sgn x)^\alpha u_1^\alpha (vQ_\alpha BQ_{0}v)(u_1,u_2).
\end{aligned}
$$
Similarly by decomposition \eqref{decomposition_low}, it suffices to study the following integral kernels for each $N \leq N^\prime$ and sum up:
$$
\begin{aligned}
 \int_{\R^2\times[0,1]}M_{\alpha0}(X_1,\Omega_1) \Big(\int_0^\infty  e^{-i t \sqrt{\lambda^4+m^2}}e^{\pm i\lambda(|X_1|+|y-u_2|)}
\lambda^{\gamma } \widetilde{\chi}_1(\lambda) \varphi_0(2^{-N} \lambda) f_{\alpha 0}^\pm d\lambda \Big)d\Omega_1.
\end{aligned}
$$
Note that $M_{\alpha0}$ and $f_{\alpha 0}^\pm$ satisfy the same estimates as \eqref{proposition_regular} and \eqref{f} for $M_{\alpha\beta}$ and $f_{\alpha \beta}^\pm$, respectively.
Then by the same argument as in case (i), we can obtain the same estimates for $K^\pm$ when $\alpha \neq 0, \beta=0$.

{\it Case (iii)}. Let $\alpha=0$, $\beta\neq0$. Again, we rewrite $K^{\pm}$ in the form
$$
\int_{\R^2\times[0,1]}M_{0\beta}(Y_2,\Omega_2) \left(\int_0^\infty e^{-i t \sqrt{\lambda^4+m^2}}e^{\pm i\lambda(|x-u_1|+|Y_2|)}\lambda^{\gamma } \widetilde{\chi}_1(\lambda) f_{0\beta}^\pm(\lambda,m,\ell,x-u_1,Y_2) d\lambda \right)d\Omega_2,
$$
where $\Omega_2=(u_1,u_2,\theta_2)$, $Y_2=y-\theta_2u_2$ and
$$
\begin{aligned}
&f_{0 \beta}^\pm(\lambda,m,\ell,x-u_1,Y_2)=\left(\lambda^4+m^2\right)^{\frac{\ell}{2}}\mathcal{F}_{0}^{ \pm}\left(\lambda|x-u_1|\right)\mathcal{F}_{\beta}^{\pm}\left(\lambda|Y_2|\right), \\
&M_{0\beta}(y,\Omega_2)=\frac{(-1)^\beta}{16(\beta-1)!}(1-\theta_2)^{\beta-1}(\sgn y)^\beta u_2^\beta (vQ_0 BQ_{\beta}v)(u_1,u_2).
\end{aligned}
$$
Then the same argument as above  implies the desired conclusions. This completes the proof.
\end{proof}

Now it remains to deal with the following integral kernels $K_3^{0,\pm}$, which are given by the following Proposition \ref{prop_Gamma}.
$$
\begin{aligned}
 K^{0,\pm}_{3}(m,\ell,t,x,y)=\int_0^\infty  e^{-i t \sqrt{\lambda^4+m^2}}\lambda^3 \left(\lambda^4+m^2\right)^{\frac{\ell}{2}} \widetilde{\chi}_1(\lambda)\left[R_0^\pm(\lambda^4)v\Gamma_3^0(\lambda)vR_0^{ \pm}(\lambda^4)\right](x,y)d\lambda.
\end{aligned}
$$
\begin{proposition}\label{prop_Gamma}
Let $K_3^{0,\pm}(m,\ell,t, x, y)$ be defined above, then
$$
\begin{aligned}
\sup _{x, y \in \mathbb{R}}\left| K_3^{0,\pm}(m,\ell,t, x, y)\right| \lesssim
\begin{cases}
  C_\ell |t|^{-\frac{1+2\ell}{2}}, & \text{if } m=0,-\frac{1}{2}<\ell\leq0, \\
  m^{\ell} \left\langle  m \right\rangle^{\frac{1}{2}}|t|^{-\frac{1}{4}}, & \text{if } m \neq 0,\ell\leq0.
\end{cases}
\end{aligned}
$$
\end{proposition}
\begin{proof}
The proof is more involved than in the previous case since $\Gamma^0_3$ depends on $\lambda$.
By decomposition \eqref{decomposition_low}, it suffices to deal with the following integral kernels $K_{3, N}^{0, \pm}(m, \ell,t, x, y)$ for each $N \leq N^\prime$ and sum up:
$$
\begin{aligned}
& \int_0^{\infty} e^{-it \sqrt{\lambda^4+m^2}} \lambda^3 \left(\lambda^4+m^2\right)^{\frac{\ell}{2}} \widetilde{\chi}_1(\lambda)\varphi_0(2^{-N}\lambda)\left\langle[v \Gamma_3^0(\lambda) v]\left(R_0^{ \pm}(\lambda^4)(*, y)\right)(\cdot), R_0^{\mp}(\lambda^4)(x, \cdot)\right\rangle d\lambda.
\end{aligned}
$$
Note that
$$
R_0^{\pm}(\lambda^4)(x, y)=\frac{1}{4 \lambda^3}\left( \pm ie^{\pm i \lambda|x-y|}-e^{-\lambda|x-y|}\right)=\frac{e^{ \pm i \lambda|x-y|}}{4 \lambda^3} \mathcal{F}_{0}^{\pm}\left(\lambda|x-y|\right),
$$
thus, we have
$$
\begin{aligned}
& \left\langle [v \Gamma_3^0(\lambda) v]\left(R_0^{ \pm}(\lambda^4)(*, y)\right)(\cdot), R_0^{\mp}(\lambda^4)(x, \cdot)\right\rangle \\
&\quad = \frac{1}{16 \lambda^6}\left\langle[v \Gamma_3^0(\lambda) v]\left(e^{ \pm i \lambda|*-y|} \mathcal{F}_{0}^{\pm}\left(\lambda|*-y|\right)\right)(\cdot),\left(e^{\mp i \lambda|x-\cdot|} \mathcal{F}_{0}^{\pm}\left(\lambda|x-\cdot|\right)\right)\right\rangle \\
&\quad =\frac{1}{16 \lambda^6} e^{ \pm i \lambda|x|} e^{ \pm i \lambda|y|}\left\langle[v \Gamma_3^0(\lambda) v]\left(e^{ \pm i \lambda(|*-y|-|y|)} \mathcal{F}_{0}^{\pm}\left(\lambda|*-y|\right)\right)(\cdot),\left(e^{\mp i \lambda(|x-\cdot|-|x|)} \mathcal{F}_{0}^{\pm}\left(\lambda|x-\cdot|\right)\right)\right\rangle \\
&\quad :=\frac{1}{16 \lambda^6} e^{ \pm i \lambda(|x|+|y|)} E_{3,N}^{0,\pm}(\lambda,x,y).
\end{aligned}
$$
Utilizing the change of variable $\lambda \rightarrow 2^Ns$, $K_{3, N}^{0, \pm}$ can be rewriten as
$$
\begin{aligned}
\frac{2^{-2 N}}{16} \int_0^{\infty} e^{-i t \sqrt{2^{4N} s^4+m^2}} s^{-3} \left(2^{4N} s^4+m^2\right)^{\frac{\ell}{2}} \widetilde{\chi}_1(2^N s)\varphi_0(s) e^{ \pm i 2^N s(|x|+|y|)} E_{3, N}^{0, \pm}\left(2^N s, x, y\right) ds.
\end{aligned}
$$
Note that for any $s \in \operatorname{supp} \varphi_0(s), k=0,1$,
$$
\begin{aligned}
&\left|\partial_s^k\left(e^{ \pm i 2^N s(|*-y|-|y|)} \mathcal{F}_{0}^{\pm}\left(2^N s|*-y|\right)\right)\right| \lesssim \left\langle *\right\rangle^k, \\
&\left|\partial_s^k\left(e^{ \pm i 2^N s(|x-\cdot|-|x|)} \mathcal{F}_{0}^{\pm}\left(2^N s|x-\cdot|\right)\right)\right| \lesssim \left\langle\cdot\right\rangle^k,
\end{aligned}
$$
and
$
\left\|\partial_s^k\left(\Gamma_3^0(2^N s)\right)\right\|_{L^2 \rightarrow L^2} \lesssim 2^{3N} s^{3-k}
$ from \eqref{lem_Gambda},
then by H\" older inequality, we have
$$
\begin{aligned}
\left|\partial_s^k E_{3,N}^{0,\pm}\left(2^N s, x,y\right)\right| \lesssim \sum_{k=0}^1 \left\|v(\cdot)\langle\cdot\rangle^{1-k}\right\|_{L^2}^2 \left\|\partial_s^k\left(\Gamma_3^0(2^N s)\right)\right\|_{L^2 \rightarrow L^2} \lesssim 2^{3N} s^{3-k}.
\end{aligned}
$$
Let $z=\left(x, y\right)$ and $\Psi(z)=|x|+|y|,$
$$
\begin{aligned}
\Phi^{ \pm}(2^N s, m,\ell, z)=
\begin{cases}
  2^{-3N} \widetilde{\chi}_1(2^N s)E_{3, N}^{0, \pm}(2^N s, x, y), & \text{if } m=0,-\frac{1}{2}<\ell\leq0,\\
  2^{-3N}m^{-\ell}\left(2^{4N}s^4+m^2\right)^{\frac{\ell}{2}}\widetilde{\chi}_1(2^N s) E_{3, N}^{0, \pm}(2^N s, x, y), & \text{if } m \neq 0,\ell\leq0,
\end{cases}
\end{aligned}
$$
then by Lemma \ref{estimate-integral} and Remark \ref{estimate-integral_remark}, we obtain that
$$
\sup _{x, y \in \mathbb{R}}\left|K_{3, N}^{0, \pm}(m,\ell,t, x,y)\right| \lesssim
\begin{cases}
  2^{-4\ell} 2^{(1+2\ell)N}\Theta_{N_0, N}(0,t), & \text{if } m=0,-\frac{1}{2}<\ell\leq0, \\
  2^{ N} m^{\ell} \left\langle  m \right\rangle^{\frac{1}{2}}\left(1+|t|2^{4N}\right)^{-\frac{1}{2}}, & \text{if } m \neq0, \ell\leq 0.
\end{cases}
$$
Finally, by Lemma \ref{sum} (i) and (iii), we obtain the desired conclusions.
\end{proof}
With Propositions \ref{free_case}--\ref{free_case_m} and \ref{prop_K}--\ref{prop_Gamma} at hand, we immediately get the corresponding low energy decay estimates of
 $$\cos(t\sqrt{H})\ \text{and}\ \cos(t\sqrt{H+m^2}),~  \frac{\sin(t\sqrt{H+m^2})}{\sqrt{H+m^2}},~m\neq0,
 $$
 in Theorem \ref{main_theorem_low}. And note that
 $$
 \left\|\frac{\sin (t \sqrt{H})}{\sqrt{H}} P_{a c}(H)\chi_1(H) \right\|_{L^1 \rightarrow L^{\infty}} \lesssim \int_{-t}^{t}\left\|\cos(s\sqrt{H})P_{a c}(H)\chi_1(H)\right\|_{L^1 \rightarrow L^{\infty}}ds \lesssim |t|^{\frac{1}{2}}.
 $$
 Therefore, the proof of Theorem \ref{main_theorem_low} for the regular case is completed.
\subsection{The first kind of resonance}
Next we consider the case when zero is the first kind resonance of $H$ and $|V(x)|\lesssim \<x\>^{-\mu}$ with $\mu>17$.
Substituting the expansion \eqref{lem_M_02} into \eqref{id-RV}, we obtain
$$
\begin{aligned}
R_V^{ \pm}(\lambda^4)= &R_0^\pm(\lambda^{4})
 -R_0^\pm(\lambda^{4})v\Big\{\lambda^{-1}Q_{2}^{0}A_{-1}^1Q_{2}^{0}+Q_{2}A_{01}^1Q_1+Q_1A_{02}^1Q_{2}
 +\lambda\left(Q_1A_{11}^1Q_1
\right.\\
&\left. +Q_{2}A_{12}^1+A_{13}^1Q_{2}\right)
 +\lambda^2\left(Q_1A_{21}^1+A_{22}^1Q_1\right)+\Gamma_3^1(\lambda)\Big\}vR_0^\pm(\lambda^{4}).
\end{aligned}
$$
Inserting this formula into \eqref{stone_low_1}, and combining with Propositions \ref{free_case} and \ref{free_case_m}, we only need to consider the integral kernels \eqref{formA_1}
with
$$
A=\lambda^{-1} Q_{2}^{0}A_{-1}^1Q_{2}^{0},~Q_{2}A_{01}^1Q_1,~Q_1A_{02}^1Q_{2},~\lambda Q_{2}A_{12}^1,
~\lambda A_{13}^1Q_{2},~\lambda^2 Q_1A_{21}^1,~\lambda^2 A_{22}^1Q_1,
$$
while the cases with $A=\lambda Q_1A_{11}^1Q_1, \Gamma_3^1(\lambda)$ can be dealt with by the same method as in the regular case with $A=\lambda Q_1A_{1}^0 Q_1, \Gamma_3^0(\lambda)$.
And we can check that the corresponding kernels can be written in following form:
\begin{equation}\label{form_1}
\begin{split}
\int_0^\infty e^{-i t \sqrt{\lambda^4+m^2}}\lambda^{6-\alpha-\beta}\left(\lambda^4+m^2\right)^{\frac{\ell}{2}}
\widetilde{\chi}_1(\lambda)\left[R_0^\pm(\lambda^4)vQ_\alpha BQ_\beta vR_0^\pm(\lambda^4)\right](x,y)d\lambda,
\end{split}
\end{equation}
with $Q_\alpha BQ_\beta \in \mathbb{AB}(L^2)$.

Hence the same proof as that of Proposition \ref{prop_K} yields the same bounds for the above integral kernels, which implies the results of Theorem \ref{main_theorem_low} for the first kind resonance case.
\subsection{The second kind of resonance}
When zero is the second kind resonance of $H$ and $|V(x)|\lesssim \<x\>^{-\mu}$ with $\mu>25$.
Substituting the expansion \eqref{lem_M_03} into \eqref{id-RV}, we obtain
$$
\begin{aligned}
R_V^{ \pm}(\lambda^4)= &R_0^\pm(\lambda^{4})
 -R_0^\pm(\lambda^{4})v\Big\{\lambda^{-3}Q_{3}A_{-3}^2Q_{3}+\lambda^{-2}\left(Q_{3}A_{-21}^2Q_{2}
 +Q_{2}A_{-22}^2Q_{3}\right)+\lambda^{-1}\left(Q_{2}A_{-11}^2Q_{2}\right.\\
&\left.+Q_{3}A_{-12}^2Q_1+Q_1A_{-13}^2Q_{3}\right) +Q_{2}A^2_{01}Q_1+Q_1A^2_{02}Q_{2}+Q_{3}A^2_{03}+A^2_{04}Q_{3}+\lambda\left(Q_1A_{11}^2Q_1\right.\\
&\left.+Q_{2}A^2_{12}+A^2_{13}Q_{2}\right)
+\lambda^2\left(Q_1A^2_{21}+A^2_{22}Q_1\right)
+\Gamma_3^2(\lambda)\Big\}vR_0^\pm(\lambda^{4}).
\end{aligned}
$$
Combining with argument in the previous case, we only need to consider the integral kernels \eqref{formA_1} with
$$
\small
\begin{aligned}
A=\lambda^{-3}Q_{3}A_{-3}^2Q_{3},~\lambda^{-2}Q_{3}A_{-21}^2Q_{2}, ~\lambda^{-2}Q_{2}A_{-22}^2Q_{3}, ~\lambda^{-1}Q_{3}A_{-12}^2Q_1, ~\lambda^{-1}Q_1A_{-13}^2Q_{3},
~Q_{3}A^2_{03}, ~A^2_{04}Q_{3}.
\end{aligned}
$$
And it can be checked that the corresponding kernels can also be written in the form \eqref{form_1}.

Therefore, the same proof used in Proposition \ref{prop_K} leads to identical bounds for the integral kernels above, which implies the conclusions of Theorem \ref{main_theorem_low} for the second kind resonance case.
\section{High energy decay estimates}\label{sec:high-energy}
Here we give the proof of the high energy part of Theorem \ref{main_theorem}, that is, the following Theorem \ref{main_theorem_high}.
By using Stone's formula,
\begin{align}
\nonumber
 (H&+m^2)^{\frac{\ell}{2}}e^{-i t \sqrt{H+m^2}}P_{ac}(H)\chi_2(H)    \\
 \label{stone_high_1}
&= \frac{2}{\pi i} \int_0^{\infty} e^{-it \sqrt{\lambda^4+m^2}}\lambda^3 \left(\lambda^4+m^2\right)^{\frac{\ell}{2}}\widetilde{\chi}_2(\lambda) [R_V^{+}(\lambda^4)-R_V^{-}(\lambda^4)](x,y) d \lambda,
\end{align}
where $\widetilde{\chi}_2(\lambda)=\chi_2(\lambda^4)(\lambda>0)$ so that $\operatorname{supp}\widetilde{\chi}_2 \subset [\lambda_0^{\frac{1}{4}},+\infty)$.
And note that
$$
\begin{aligned}
\frac{\sin (t \sqrt{H})}{\sqrt{H}} P_{a c}(H)\chi_2(H) =\frac{1}{2}\int_{-t}^{t}\cos(s\sqrt{H})P_{a c}(H)\chi_2(H)ds.
\end{aligned}
$$
Then it is sufficient to consider \eqref{stone_high_1} when $m=0,\ell=0$ and $m\neq0,\ell=-1,0.$
\begin{theorem}\label{main_theorem_high}
Let $|V(x)| \lesssim(1+|x|)^{-2-}$. Assume that $H=\Delta^2+V$ has no positive embedded eigenvalue and $P_{ac}(H)$ denotes the projection onto the absolutely continuous spectrum of $H$, then
$$
\begin{aligned}
&\left\|e^{-it \sqrt{H}} P_{ac}(H) \chi_2(H)\right\|_{L^1 \rightarrow L^{\infty}} \lesssim|t|^{-\frac{1}{2}},\\
&\left\|\frac{\sin (t \sqrt{H})}{\sqrt{H}} P_{ac}(H) \chi_2(H)\right\|_{L^1 \rightarrow L^{\infty}} \lesssim|t|^{\frac{1}{2}},
\end{aligned}
$$
and for $m \neq 0$,
$$
\begin{aligned}
&\left\|\cos (t \sqrt{H+m^2}) P_{a c}(H) \chi_2(H)\right\|_{L^1 \rightarrow L^{\infty}} \lesssim \left\langle  m \right\rangle |t|^{-\frac{1}{2}},\\
&\left \|\frac{\sin (t \sqrt{H+m^2})}{\sqrt{H+m^2}} P_{a c}(H) \chi_2(H)\right \|_{L^1 \rightarrow L^{\infty}}
\lesssim (1+\frac{1}{m})|t|^{-\frac{1}{2}}.
\end{aligned}
$$
\end{theorem}
In this section, we use the following resolvent identity:
$$
R_V^{\pm}(\lambda^4)=R_0^{\pm}(\lambda^4)-R_0^{ \pm}(\lambda^4)VR_0^{\pm}(\lambda^4)+R_0^{ \pm}(\lambda^4)VR_V^{\pm}(\lambda^4)VR_0^{ \pm}(\lambda^4).
$$
Substituting this formula into \eqref{stone_high_1}, and combining with Propositions \ref{free_case} and \ref{free_case_m}, it is sufficient to deal with the following integral kernels, which are given by the following Propositions \ref{high_1} and \ref{high_2}, respectively:
$$
\begin{aligned}
\nonumber
& L_{1}^{\pm}(m,\ell,t,x,y):=\int_0^{\infty} e^{-i t \sqrt{\lambda^4+m^2}} \lambda^3 \left(\lambda^4+m^2\right)^{\frac{\ell}{2}} \widetilde{\chi}_2(\lambda)[R_0^{ \pm}(\lambda^4)VR_0^{ \pm}(\lambda^4)](x, y) d\lambda, \\
&L_{2}^{\pm}(m,\ell,t,x,y):=\int_0^{\infty} e^{-i t \sqrt{\lambda^4+m^2}} \lambda^3 \left(\lambda^4+m^2\right)^{\frac{\ell}{2}} \widetilde{\chi}_2(\lambda)[R_0^{\pm}(\lambda^4)VR_V^{\pm}(\lambda^4)VR_0^{\pm}(\lambda^4)](x, y) d\lambda.
\end{aligned}
$$
Moreover, by Littlewood-Paley decomposition \eqref{varphi_0}, we can decompose $\widetilde{\chi}_i$ as follows:
\begin{equation}
\label{decomposition_high}
\widetilde{\chi}_2(\lambda)=\sum_{N=N^{\prime\prime}}^{+\infty}\widetilde{\chi}_2(\lambda)
\varphi_0(2^{-N}\lambda),\quad \lambda>0,
\end{equation}
where we may take $N^{\prime\prime}=\left \lfloor\frac{1}{4} \log_2\lambda_0\right \rfloor$ since $ \operatorname{supp} \widetilde{\chi}_2\subset [\lambda_0^{\frac{1}{4}},+\infty), \operatorname{supp} \varphi_0(2^{-N}\lambda) \subset [2^{N-2},2^N]$.
\begin{proposition}\label{high_1}
Assume that $|V(x)| \lesssim \left \langle  x\right \rangle^{-1-}$. Then
$$
\begin{aligned}
\sup _{x, y \in \mathbb{R}}\left|L_{1}^{\pm}(m,\ell,t,x,y)\right| \lesssim
\begin{cases}
  C_\ell |t|^{-\frac{1+2\ell}{2}}, & \text{if } m=0,-\frac{1}{2}<\ell\leq0, \\
  m^{\ell} \left\langle  m \right\rangle |t|^{-\frac{1}{2}}, & \text{if } m\neq0,\ell\leq0.
\end{cases}
\end{aligned}
$$
\end{proposition}
\begin{proof}
 By decomposition \eqref{decomposition_high}, it is enough to analyze the following integrals $L_{1,N}^{\pm}(m,\ell,t,x,y)$ for each $N \geq N^{\prime\prime}$ and sum up:
$$
\begin{aligned}
\int_0^{\infty} e^{-it\sqrt{\lambda^4+m^2}}\lambda^3 \left(\lambda^4+m^2\right)^{\frac{\ell}{2}} \widetilde{\chi}_2(\lambda)\varphi_0(2^{-N}\lambda) \left[R_0^{ \pm}(\lambda^4)VR_0^{\pm}(\lambda^4)\right](x,y) d\lambda.
\end{aligned}
$$
Indeed,
$$
R_0^{\pm}(\lambda^4)(x, y)=\frac{1}{4 \lambda^3}\left( \pm ie^{\pm i \lambda|x-y|}-e^{-\lambda|x-y|}\right)=\frac{e^{ \pm i \lambda|x-y|}}{4 \lambda^3} \mathcal{F}_{0}^{\pm}\left(\lambda|x-y|\right).
$$
And applying the change of variable $\lambda \rightarrow 2^Ns$, we can rewrite $L_{1,N}^{\pm}$ as
\begin{equation}\label{L}
\begin{split}
& \frac{2^{-2N}}{16} \int_{\mathbb{R}}\Big(\int_0^{\infty} e^{-i t \sqrt{2^{4N}s^4+m^2}} s^{-3} \left(2^{4N}s^4+m^2\right)^{\frac{\ell}{2}} \widetilde{\chi}_2(2^N s)\varphi_0(s) e^{\pm i 2^N s|x-y_1|}  \\
& e^{\pm i 2^N s|y-y_1|} f_{0 0}^\pm(2^N s, x-y_1,y-y_1) V(y_1) ds\Big)dy_1,
\end{split}
\end{equation}
where $f_{0 0}^\pm(2^N s, x-y_1,y-y_1)=\mathcal{F}_0^{\pm}(2^N s|x-y_1|)\mathcal{F}_0^{\pm}(2^N s|y-y_1|).$

Note that for any $s \in \operatorname{supp} \varphi_0(s)$ and $k=0,1$,
$$
\begin{aligned}
\sup_{x, y \in \mathbb{R}}\left|\partial_s^k\left(f_{0 0}^\pm(2^N s, x-y_1,y-y_1)\right)\right| \lesssim 1, \quad
\left|\partial_s^k \left(m^{-\ell} \left(2^{4N}s^4+m^2\right)^{\frac{\ell}{2}}\right)\right| \lesssim 1,~m \neq 0,\ell\leq 0.
\end{aligned}
$$
Let $z=(x,y,y_1)$ and $\Psi(z)=|x-y_1|+|y-y_1|,$
$$
\begin{aligned}
\Phi^{ \pm}(2^N s, m, \ell,z)=
\begin{cases}
\widetilde{\chi}_2(2^N s) f_{0 0}^\pm(2^N s, x-y_1,y-y_1),&\text{if } m=0,-\frac{1}{2}<\ell\leq 0, \\
m^{-\ell}\left(2^{4N}s^4+m^2\right)^{\frac{\ell}{2}}\widetilde{\chi}_2(2^N s) f_{0 0}^\pm(2^N s, x-y_1,y-y_1),&\text{if } m \neq 0,\ell\leq0,
\end{cases}
\end{aligned}
$$
then by Lemma \ref{estimate-integral} and Remark \ref{estimate-integral_remark}, we obtain that the integral with respect to $s$ in \eqref{L} is controlled by
$$
\begin{aligned}
\begin{cases}
  2^{-4\ell} 2^{(-2+2\ell)N}\Theta_{N_0, N}(0,t), & \text{if } m=0,-\frac{1}{2}<\ell\leq 0 \\
  2^{-2N} m^{\ell}\Theta_{N_0, N}(m,t), & \text{if } m \neq 0,\ell\leq0.
\end{cases}
\end{aligned}
$$
Since $V(x) \lesssim \left \langle  x\right \rangle^{-1-}$, then for $m = 0,-\frac{1}{2}<\ell\leq0$,
$$
\left|L_{1,N}^{\pm}(0,\ell,t,x,y)\right| \lesssim 2^{-4\ell} 2^{(-2+2\ell)N}\Theta_{N_0, N}(0,t)\int_{\mathbb{R}} |V(y_1)|  d y_1 \lesssim 2^{-4\ell} 2^{(-2+2\ell)N}\Theta_{N_0, N}(0,t),
$$
and for $m \neq 0,\ell\leq 0$,
$$
\left|L_{1,N}^{\pm}(m,\ell,t,x,y)\right| \lesssim 2^{-2N}m^{\ell}\Theta_{N_0, N}(m,t)\int_{\mathbb{R}} |V(y_1)|  d y_1 \lesssim 2^{-2N}m^{\ell}\Theta_{N_0, N}(m,t).
$$

Finally, by Lemma \ref{sum} (i) and (ii), we immediately get the desired conclusions.
\end{proof}

We next consider the following integral kernels:
$$
L_{2}^{\pm}(m,\ell,t,x,y):=\int_0^{\infty} e^{-i t \sqrt{\lambda^4+m^2}} \lambda^3 \left(\lambda^4+m^2\right)^{\frac{\ell}{2}} \widetilde{\chi}_2(\lambda)[R_0^{\pm}(\lambda^4)VR_V^{\pm}(\lambda^4)VR_0^{\pm}(\lambda^4)](x, y) d\lambda.
$$ To deal with these integral kernels, we utilize the following estimates.
\begin{lemma}[{\cite[Theorem 2.23]{FSY}}]\label{RRR}
Let $k \geq 0$ and $|V(x)| \lesssim(1+|x|)^{-k-1}$ such that $H=\Delta^2+V$ has no embedded positive eigenvalues. Then, for any $\sigma>k+\frac{1}{2}, R_V^{ \pm}(\lambda) \in \mathcal{B}\left(L_\sigma^2(\mathbb{R}), L_{-\sigma}^2(\mathbb{R})\right)$ are $C^k$-continuous for all $\lambda>0$. Furthermore,
$$
\left\|\partial_\lambda^k R_V^{ \pm}(\lambda)\right\|_{L_\sigma^2(\mathbb{R}) \rightarrow L_{-\sigma}^2(\mathbb{R})}=O\left(|\lambda|^{-\frac{3(k+1)}{4}}\right),~\lambda \rightarrow+\infty.
$$
\end{lemma}
\begin{proposition}\label{high_2}
Assume that $|V(x)| \lesssim(1+|x|)^{-2-}$. Then
$$
\begin{aligned}
\sup _{x, y \in \mathbb{R}}\left|L_{2}^{\pm}(m,\ell,t,x,y)\right| \lesssim
\begin{cases}
  C_\ell |t|^{-\frac{1+2\ell}{2}}, & \text{if } m=0,-\frac{1}{2}<\ell\leq0, \\
  m^{\ell} \left\langle  m \right\rangle |t|^{-\frac{1}{2}}, & \text{if } m\neq0,\ell\leq0.
\end{cases}
\end{aligned}
$$
\end{proposition}
\begin{proof}
 By decomposition \eqref{decomposition_high}, it suffices to deal with the following integrals $L_{2,N}^{\pm}(m,\ell,t,x,y)$ for each $N \geq N^{\prime\prime}$ and sum up:
$$
\begin{aligned}
\int_0^{\infty} e^{-it \sqrt{\lambda^4+m^2}} \lambda^3 \left(\lambda^4+m^2\right)^{\frac{\ell}{2}} \widetilde{\chi}_2(\lambda) \varphi_0(2^{-N} \lambda)
\left\langle V R_V^{ \pm}(\lambda^4) V\left(R_0^{ \pm}\left(\lambda^4\right)(*, y)\right)(\cdot),\left(R_0^{ \pm}(\lambda^4)\right)^*(x, \cdot)\right\rangle d\lambda.
\end{aligned}
$$
Indeed,
$$
R_0^{\pm}(\lambda^4)(x, y)=\frac{1}{4 \lambda^3}\left( \pm ie^{\pm i \lambda|x-y|}-e^{-\lambda|x-y|}\right)=\frac{e^{ \pm i \lambda|x-y|}}{4 \lambda^3} \mathcal{F}_{0}^{\pm}\left(\lambda|x-y|\right).
$$
It follows that
$$
\begin{aligned}
& \left\langle VR_V^{ \pm}(\lambda^4) V\left(R_0^{\pm}(\lambda^4)(*, y)\right)(\cdot), R_0^{\mp}(\lambda^4)(x, \cdot)\right\rangle \\
&= \frac{1}{16 \lambda^6}\left\langle V R_V^{\pm} V\left(e^{\pm i \lambda|*-y|} \mathcal{F}_0^{\pm}(\lambda|*-y|)\right)(\cdot),\left(e^{\mp i \lambda|x-\cdot|} \mathcal{F}_0^{\mp}(\lambda|x-\cdot|)\right)\right\rangle \\
&= \frac{1}{16 \lambda^6} e^{\pm i \lambda|x|} e^{\pm i \lambda|y|}\left\langle V R_V^{\pm} V\left(e^{\pm i \lambda(|*-y|-|y|)} \mathcal{F}_0^{\pm}(\lambda|*-y|)\right)(\cdot)\left(e^{\mp i \lambda(|x-\cdot|-|x|)} \mathcal{F}_0^{\mp}(\lambda|x-\cdot|)\right)\right\rangle \\
&:= \frac{1}{16 \lambda^6} e^{\pm i \lambda(|x|+|y|)} E_{2,N}^{\pm}(\lambda,x,y).
\end{aligned}
$$
Utilizing the change of variable $\lambda \rightarrow 2^N s$, we can rewrite $L_{2,N}^{\pm}(m,\ell,t,x,y)$ as
$$
\begin{aligned}
\frac{2^{-2N}}{16} \int_0^{\infty} e^{-it \sqrt{2^{4N} s^4+m^2}} e^{\pm i 2^N s(|x|+|y|)}  s^{-3} \widetilde{\chi}_2(2^N s)\varphi_0(s)
 \left(2^{4N}s^4+m^2\right)^{\frac{\ell}{2}}  E_{2,N}^{\pm}(2^N s , x, y) ds.
\end{aligned}
$$
Note that for any $s \in \operatorname{supp} \varphi_0(s),k=0,1$,
$$
\begin{aligned}
& \left|\partial_s^k\left(e^{\pm i 2^N s(|*-y|-|y|)} \mathcal{F}_0^{\pm}(2^N s|*-y|)\right)\right| \lesssim 2^{kN}\langle *\rangle^k, \\
& \left|\partial_s^k\left(e^{ \pm i 2^N s(|x-\cdot|-|x|)} \mathcal{F}_0^{\pm}(2^N s)(x, \cdot)\right)\right| \lesssim 2^{kN}\langle\cdot\rangle^k.
\end{aligned}
$$
By Lemma \ref{RRR}, for $\sigma>k+\frac{1}{2}$ with $k=0,1$, we obtain
$$
\left\|\partial_s^k R_V^{ \pm}(2^{4N} s^4)\right\|_{L_\sigma^2 \rightarrow L_{-\sigma}^2} \lesssim 2^{kN}\left(2^N s\right)^{-3}.
$$
Furthermore, since $|V(x)| \lesssim \left \langle  x\right \rangle^{-2-}$, by H\"older's inequality, then for $\sigma>k+\frac{1}{2}$ with $k=0,1$,
$$
\begin{aligned}
\left|\partial_s^k E_{2, N}^{\pm}(2^Ns,x,y)\right|
& \lesssim \sum_{k=0}^12^{(1-k)N} \left\|V(\cdot)\langle\cdot\rangle^{\sigma+1-k}\right\|_{L^2}^2 \cdot\left\|\partial_s^k R_V^{\pm}(2^{4 N} s^4)\right\|_{L_\sigma^2 \rightarrow L_{-\sigma}^2}  \lesssim 2^{N}\left(2^N s\right)^{-3}.
\end{aligned}
$$
Note that for $m \neq 0,\ell\leq0$,
$$
\left|\partial_s^k \left(m^{-\ell} \left(2^{4N}s^4+m^2\right)^{\frac{\ell}{2}}\right)\right| \lesssim 1.
$$
Let $z=(x, y)$ and $\Psi(z)=|x|+|y|,$
$$
\begin{aligned}
\Phi^{ \pm}(2^N s, m,\ell, z)=
\begin{cases}
\widetilde{\chi}_2(2^N s)E_{2,N}^{\pm}(2^N s,m,x,y), &\text{if } m=0,-\frac{1}{2}<\ell\leq 0, \\
m^{-\ell} \left(2^{4N}s^4+m^2\right)^{\frac{\ell}{2}}\widetilde{\chi}_2(2^N s)E_{2,N}^{\pm}(2^N s,m,x,y), &\text{if } m \neq 0,\ell\leq0,
\end{cases}
\end{aligned}
$$
then by Lemma \ref{estimate-integral} and Remark \ref{estimate-integral_remark}, we have
$$
\left|L_2^{\pm}(m,\ell,t,x,y)\right| \lesssim
\begin{cases}
  2^{-4\ell} 2^{(-4+2\ell)N}\Theta_{N_0, N}(0,t), & \text{if } m=0,-\frac{1}{2}<\ell\leq0, \\
  2^{-4N}m^{\ell}\Theta_{N_0, N}(m, t), & \text{if } m\neq 0,\ell\leq0
\end{cases}
$$
uniformly in $x, y$.

Finally, by Lemma \ref{sum} (i) and (ii), we obtain the desired conclusions.
\end{proof}

With Propositions \ref{free_case}--\ref{free_case_m}, \ref{high_1} and \ref{high_2} at hand, we immediately get the corresponding high energy decay estimates of
 $$\cos(t\sqrt{H}),\ \ \cos(t\sqrt{H+m^2}) \  \text{and}\  \frac{\sin(t\sqrt{H+m^2})}{\sqrt{H+m^2}},~m\neq0,
 $$
 in Theorem \ref{main_theorem_high}. And note that
 $$
 \left\|\frac{\sin (t \sqrt{H})}{\sqrt{H}} P_{a c}(H)\chi_2(H) \right\|_{L^1 \rightarrow L^{\infty}} \lesssim \int_{-t}^{t}\left\|\cos(s\sqrt{H})P_{a c}(H)\chi_2(H)\right\|_{L^1 \rightarrow L^{\infty}}ds \lesssim |t|^{\frac{1}{2}}.
 $$
 Therefore, the proof of Theorem \ref{main_theorem_high} is completed.

\bigskip


{\bf Acknowledgements:}
S. Chen, Z. Wan and X. Yao are partially supported by NSFC grants No.11771165 and 12171182. The authors would like to express their thanks to Professor Avy Soffer for his interests and insightful discussions about topics on higher-order operators.


\end{document}